\documentclass{amsart}

\usepackage[hidelinks]{hyperref}
\usepackage{cleveref}

\newcommand{\ddw}[1]{\frac{\partial #1}{\partial w}}
\newcommand{\ddy}[1]{\frac{\partial #1}{\partial y}}

\newcommand{\ddx}[1]{\frac{\partial #1}{\partial x}}
\newcommand{\ddt}[1]{\frac{\partial #1}{\partial t}}

\newcommand{\twoddw}[1]{\frac{\partial^2 #1}{\partial w^2}}
\newcommand{\twoddx}[1]{\frac{\partial^2 #1}{\partial x^2}}

\newcommand{\dd}[2]{\frac{\partial #1}{\partial #2}}
\newcommand{\ddtwo}[2]{\frac{\partial^2 #1}{\partial #2^2}}

\newcommand{\totdd}[2]{\frac{d #1}{d #2}}
\newcommand{\tottwodd}[2]{\frac{d^2 #1}{d {#2}^2}}

\newcommand{\ddNorm}[1]{\frac{\partial #1}{\partial \nu}}

\newcommand{\frechet}[2]{\frac{\delta #1}{\delta #2}}

\newcommand{\boltzmann}[1]{{}^{B}\!S[#1]}

\newcommand{\diff}[2]{D\left[#1,#2\right]}
\newcommand{\varfunc}[3]{#3\left[#1,#2\right]}

\newcommand{\divg}{\text{div}}
\newcommand{\grad}[2]{\text{grad}_{#1}#2}

\newcommand{\rmmf}{(\mathcal{M}, \mathbb{G})}
\newcommand{\mf}{\mathcal{M}}
\newcommand{\mg}{\mathbb{G}}
\newcommand{\mk}{\mathbb{K}}

\newcommand{\kotto}{\mk_\text{Otto}}

\newcommand{\wtwograd}[2]{-\dd{}{#2}\left(\rho\dd{}{#2}\frechet{#1}{\rho}\right)}

\newcommand{\wgenMobgradN}[2]{\Delta\left(#2\Delta\frechet{#1}{\rho}\right)}

\newcommand{\wtwoMobgrad}[3]{-\dd{}{#3}\left(#2\dd{}{#3}\frechet{#1}{\rho}\right)}
\newcommand{\wgenMobgrad}[3]{\ddtwo{}{#3}\left(#2\ddtwo{}{#3}\frechet{#1}{\rho}\right)}    

\newcommand{\ptwoac}{\mathcal{P}_{2,ac}(\Omega)}
\newcommand{\ptwoacarg}[1]{\mathcal{P}_{2,ac}\left(#1\right)}
\newcommand{\pac}[1]{\mathcal{P}_{ac}\left(#1\right)}

\newcommand{\posreal}{{\mathbb{R}^+}}

\newcommand{\ltwomudelta}{L^2_\mu(\Delta;\Omega)}
\newcommand{\deltatwomu}[1]{\mathbf{\Delta}^2_{#1}}

\newcommand{\htildetwo}[1]{\widetilde{H}^{2}_{#1}}
\newcommand{\htildetwoneg}[1]{\widetilde{H}^{-2}_{#1}}

\newcommand{\hdotone}[1]{\dot{H}^{1}_{#1}}
\newcommand{\hdotoneneg}[1]{\dot{H}^{-1}_{#1}}

\newcommand{\ltwo}[1]{L^{2}_{#1}}

\newcommand{\minop}{\wedge}
\newcommand{\gini}{\mathcal{G}}
\newcommand{\giniarg}[1]{\mathcal{G}\left[#1\right]}

\newcommand{\cdspace}{{\mathcal{C}_D}}
\newcommand{\cdspacearg}[1]{{\mathcal{C}_{#1}}}

\theoremstyle{plain}
\newtheorem{theorem}{Theorem}[section]
\newtheorem{lemma}[theorem]{Lemma}
\newtheorem{corollary}[theorem]{Corollary}
\newtheorem{proposition}[theorem]{Proposition}

\theoremstyle{definition}
\newtheorem{definition}[theorem]{Definition}
\newtheorem{structAssumpt}{Structural assumption}[theorem]
\crefname{structAssumpt}{structural assumption}{structural assumptions}

\newtheorem{principle}[theorem]{Principle}
\newtheorem*{principle*}{Principle}
\newtheorem*{question*}{Question}

\theoremstyle{remark}
\newtheorem{remark}[theorem]{Remark}

\numberwithin{equation}{section}

\title[Gradient structures in econophysics]{A gradient flow perspective on McKean-Vlasov equations in econophysics}

\author{David W. Cohen}
\address{Department of Mathematics, Tufts University, Medford, MA}
\email{David.Cohen@Tufts.edu}
\thanks{The author received support from the NDSEG fellowship and extends warm thanks to B.~Boghosian, M.~Johnson, D.~Gentile, C.~Smith, and B.~Mallery, all of Tufts University, for patient listening and feedback. I thank the anonymous referees for their careful reading and commentary.}



\date{January 2026}

\begin{document}

\begin{abstract}

    We prove that the Gini coefficient of economic inequality is a Lyapunov functional for a class of nonlinear, nonlocal integro-differential equations arising at the intersection of mathematics, economics, and statistical physics. Next, a novel Riemannian geometry is imposed on a subset of probability densities such that the evolutionary dynamics are formally driven by the Gini coefficient functional as a gradient flow. Thus in the same way that classical 2-Wasserstein theory connects heat flow and the Second Law of Thermodynamics by way of Boltzmann entropy, the work here gives rise to a principle of econophysics that is much of the same flavor but for the Gini coefficient.

    The noncanonical Onsager operators associated to the metric tensors are derived and some transport inequalities proven. The new metric relates to the dual norm of a second-order Sobolev-like factor space, in a similar way to how the classical 2-Wasserstein metric linearizes as the dual norm of a first-order, homogeneous Sobolev space.
\end{abstract}

\maketitle

\section{Introduction}\label{sec:intro}

We show an intuitive and new variational interpretation of a class of evolution equations, which arise from a mathematical physics treatment of idealized economic systems, that shines light on the underlying energetic structure in a unified way.

In particular, we prove the existence of a Lyapunov functional for a class of evolutionary systems, and we then introduce a novel metric geometry on a space of probability densities such that the Lyapunov functional, in fact, serves to drive the dynamics as a gradient flow.

The concept of a gradient flow of a functional over a space of probability measures is a fusing of concepts from Riemannian geometry, variational analysis, and partial differential equations. For a bounded, convex domain $\Omega\subset\mathbb{R}^n$ with smooth boundary, $\partial \Omega$, let $\mathcal{P}(\Omega)$ be the set of probability measures and $\mathcal{L}^n$ be the Lebesgue measure. The $2$-Wasserstein theory, from which we draw significant inspiration, examines the infinite-dimensional manifold
\[\ptwoac{} = \{\mu \in \mathcal{P}(\Omega)\,:\, \int_\Omega\,d\mu(x) |x|^2 < \infty \text{ and } \mu = \rho \mathcal{L}^n\}\]
and the metric
\[W_2^2(\mu,\nu) = \inf_{\text{curves }\mu_t}\left\{\int_0^1\,dt\,||\dot{\mu}_t||_{\hdotoneneg{\mu_t}}^2\,:\,\mu_0 = \mu,\, \mu_1 = \nu\right\}\] where for $\mu_t = \rho_t \mathcal{L}^n$
\[||\dot{\mu}_t||^2_{\hdotoneneg{\mu_t}} = ||\nabla \psi_t||_{\ltwo{\mu_t}}^2 \text{ where } \begin{cases}
    \divg(\rho_t\nabla\psi_t) = - \partial_t\rho_t \text{ in } \Omega\\
    \ddNorm{\psi_t} = 0 \text{ on } \partial\Omega.
\end{cases}\]

A celebrated result is that these structures induce a metric on distributions such that the linear parabolic heat flow equation, $\partial_t \rho_t = \Delta \rho_t$, is equivalent to steepest ascent of Boltzmann entropy, $\boltzmann{\rho}$. That is to say, entropy is not only increasing in time under the dynamics of the heat equation but also there is a geometry such that the dynamics may be viewed as \textit{increasing entropy in the fastest possible way} at each instant.

The classical $W_2$ theory and its nonlinear mobility generalization, $W_{2,m}$, induce structures that allow many dissipative equations to be viewed as gradient flows. See, \textit{inter alia}, \cite{MR4294651,MR2565840, MR2448650, MR2672546,MR2921215, MR3558359, MR2581977} and the references contained therein. Over the past twenty-five years, scores of well-known evolution equations have been re-formulated as Wasserstein gradient flows of functionals, often with physical meaning. Despite these abundant and elegant variational interpretations furnished to many PDE, the $W_2$ theory does have limitations.

We study a class of dissipative evolution equations that arise from the mean-field limits of interacting particle systems that model stochastic and kinetic economic systems and have proven that at least one member of this class, the yard-sale model \cite{MR3428664}, \textit{cannot} be expressed as a $W_2$ gradient flow. 

The 2-Wasserstein tangent space $T_\mu W_2(\Omega)$ is essentially the $\mu$-weighted negative first-order homogeneous Sobolev space, $\hdotoneneg{\mu}$, which is a subset of mean-zero functions. The nonlinear integro-differential evolution equation for which we furnished the nonexistence result has two conserved quantities known \textit{a priori}, namely total probability mass and total first moment. The latter is a direct consequence of the microscopic kinetics. To successfully identify an evolution within $W_2$ that conserves the first moment, it would be necessary to move only in tangent space directions that are both mean-zero and first-moment zero. That is to say \begin{equation}
    \dot{\mu}_t \in T_{\mu_t}W_2(\Omega) \cap \{ h:\Omega \rightarrow \mathbb{R} \,:\, \int_\Omega\,dx\, x h(x) = 0\} \nonumber
\end{equation}for all $t.$ This restriction on the tangent element does not seem easy to accommodate, at least in general.

In \cite{MR1842429}, F. Otto writes \begin{quote}
\textit{The merit of the right gradient flow formulation of a dissipative evolution equation is that it separates energetics and kinetics: The energetics endow the state space $M$ with a functional $E$, the kinetics endow the state space with a (Riemannian) geometry via the metric tensor $g$.}
\end{quote} In what follows, we will repeatedly have recourse to the astuteness of this observation.

We investigated alternative metric tensors that induce geometries with tangent spaces that naturally accommodate the additional known conserved quantity, just as $T_\mu W_2(\Omega)$ does for total probability mass. A primary result of this paper is a formal construction that expresses the diffusion approximations to unbiased, wealth-conserving, positivity-preserving kinetic asset exchange models as gradient flows, using a novel geometry, in a subspace of probability distributions with fixed first moment.

The modified metric tensor directly affects the interpretation of the tangent space at a distribution. In our adapted metric in one spatial dimension, the tangent elements are both mean-zero and first-moment zero functions. In this way, the additional conserved quantity is naturally incorporated into the geometry of the variational problem, in agreement with Otto's observation since the conserved quantity follows from the kinetic prescription of the process.

Observations building up in the econophysics community about the monotonicity of economic inequality in this specific class of models can now be understood in a unified fashion. Namely the class of evolution equations arising from the diffusion approximations to idealized economic models  satisfying \cref{def:positivityPreserving,def:unbiased,def:wealthConserving} are \textit{all} the gradient flows of the Gini coefficient of economic inequality, a quadratic functional of the distribution. The only difference from one model to another is the geometry placed upon the state space of wealth distributions.

These models share the same intuitive energetic principle and differ only in the specifics of their kinetics, which we relegate to the metric tensor. This formulation exactly accords with Otto's guiding beacon. In finding the correct gradient flow formulation for these models, a clear line has been drawn between their energetics and kinetics. 

The novel Riemannian metric, along with its associated manipulations, should be viewed in a purely formal light. There is some discussion about the function spaces involved but it should not be considered analytically rigorous. Likewise we generally leave aside questions of well-posedness. This decision is purposeful. The intent of this paper is to motivate, propose, and focus on the new formal gradient structures. We believe that rigorous analysis of the metric is worthwhile to pursue since the formal structures are certainly of interest.

The analysis of the \textit{diffusive transport metric}, found in \cite{MR4961432,MR4987462}, has begun to put this novel metric structure on rigorous footing. The authors of \cite{MR4961432,MR4987462} situate the diffusive transport metric, which is a one-dimensional case of the metric examined here, in a hierarchy of dynamic formulations given by the differential order of the involved continuity equations. The hierarchy begins with the Hellinger metric whose dynamic formulation is a zeroth-order continuity equation, next is the $2-$Wasserstein metric with the familiar first-order continuity equation, and finally the diffusive transport metric whose continuity equation is of second-order.

\subsection{Outline}\label{ssec:outline}
In \cref{sec:KAEM}, the econophysics models on which we focus are defined and fundamental structural assumptions are presented. A new, simple proof is given of the monotonicity of economic inequality under the evolutionary dynamics of the distribution of wealth. Next, we show that at least one model cannot be expressed as a $W_2$ gradient flow.

The novel infinite-dimensional Riemannian metric structure is defined in \cref{sec:novelMetric} and formal computations are carried out to derive the first-order functional calculus. One of the main results, \cref{thm:gradFlowEconophysics}, on expressing the econophysics models as gradient flows is proven in \cref{ssec:gradFlowKaem}. \Cref{apdx:YSMW2mDiscussion} provides heuristic motivation for the choices made in the definitions of \cref{sec:novelMetric}.

A specification of the new metric, analogous to $W_2$ being the linear mobility case of $W_{2,m}$, permits a deeper exploration of the implicated function spaces and associated PDE. Several results about the second-order Sobolev-like factor spaces are presented in \cref{sec:idenDiffCDspaces}. In addition, several transport inequalities are also proven.

Avenues of future study are suggested in \cref{sec:conclusion}.

\subsection{Notation}\label{ssec:notation}
An absolutely continuous measure $\mu$ will have density $\rho$ with respect to the Lebesgue measure, $\mathcal{L}^n$. The values of a time-dependent function at a specific time will be expressed with a subscript $t$, i.e. $\rho_t(x) = \rho(x,t)$. Derivatives with respect to time will be given explicitly as $\partial_t$ or $\ddt{}$ or $\totdd{}{t}$ or with a dot over a time-dependent function, such as $\dot{\mu_t}$. By the Fr\'echet derivative of a functional $F$, which we denote by $\frechet{F}{\rho}$, we mean the first $\ltwo{}$ variation, $\grad{\ltwo{}}{F}$. The space $\ltwo{}(\Omega)$ without explicit reference to a measure means with respect to the Lebesgue measure. Otherwise we write $\ltwo{\mu}(\Omega)$ meaning the space with inner product \[\int_\Omega\,d\mu(x)\,f(x)g(x).\] If $F$ is a functional from a space of functions $X$ over $\Omega\subset\mathbb{R}^n$ to $\mathbb{R}$ then by $\totdd{F}{t}$ there is an implicit reference to a $t$-parameterized curve $u_t: [0,T] \rightarrow X$ such that \[\totdd{F}{t} = \totdd{}{t} \left(F\circ u_t\right).\]

Some of the notation and explication follows \cite{mielke2023introduction}. Let $\rmmf$ be a Riemannian manifold where given $u\in \mf$, $\mg(u):T_u\mf\rightarrow T^*_u\mf$ is symmetric and positive. The Riemannian metric is the symmetric 2-tensor $g_u:\left(T_u\mf \times T_u\mf\right)\rightarrow\mathbb{R}$. For $v, \tilde{v} \in T_u\mf$, \begin{equation}\label{eq:defMetricFinDim}g_u(v,\tilde{v}) = \left\langle\mg(u)v, \tilde{v}\right\rangle,\end{equation} where $\langle\,,\,\rangle$ denotes the pairing between cotangent and tangent space elements. The gradient of a functional is defined via the canonical isomorphism as \begin{equation}\label{eq:gradViaCanonicalIsom}g_u\left(\grad{\mf}{\mathcal{F}(u)},v\right)=\left\langle D\mathcal{F}(u), v\right\rangle\end{equation} for all $v\in T_u\mf$, where $D\mathcal{F}$ is the differential of $\mathcal{F}$. The \textit{Onsager operator}, defined as $\mk(u):=\left(\mg(u)\right)^{-1}$, allows the expression  \[\grad{\mf}{\mathcal{F}(u)} = \left(\mk \circ D\mathcal{F}\right)(u).\]

\section{Kinetic asset exchange models and a thermodynamic second law}\label{sec:KAEM}
In this section, we briefly describe the ideas underlying econophysics and then define the class of evolution equations that we will express as intuitive gradient flows in \cref{sec:novelMetric}. Here we codify a concept that can be seen as the ``second law of econophysics,'' in analogy with the Second Law of Thermodynamics, by proving the existence of a Lyapunov functional. Next we show that a model of particular importance to the author cannot be expressed as a $W_2$ gradient flow, which serves as motivation to move beyond the $W_2$ framework.

Econophysics emerged as a field over the past three decades as the capable tools of kinetic gas theory and statistical mechanics were put to use analyzing idealized economic models \cite{MR2604625, AC2002, MR2551376, BH2002, MR2155255}. The basic recipe, though many more complicated variants exist, is that a large number of interacting agents are viewed identically except for a scalar wealth, transaction rules between two agents are prescribed (just as collisional rules are given for ideal gas particles), and the resulting many-agent system is studied via Boltzmann equations, hydrodynamic limits, diffusion approximations, and mean-field characterizations. 

These methods have had noteworthy success in producing models with few total parameters that can match real-world distributions of wealth to quite high accuracy \cite{MR4686618}.

The objects of study are continuum equations that evolve distributions of wealth forward in time from an initial condition. At time $t$, the probability of selecting an agent with wealth in the interval $[a,b]$ from a large population is $\int_a^b\,dw\, \rho_t(w)$ and the fraction of the total wealth held by agents with wealth in $[a,b]$ is $\int_a^b\,dw\, w \rho_t(w).$

\subsection{Kinetic asset exchange models}\label{ssec:kaem}
We describe a class of kinetic asset exchange models over $N$ agents each with a scalar, positive wealth that is normalized so that the mean wealth is one. Every member of the class of models to be described evolves each agent's wealth in such a way that the total wealth is constant -- this is because in each collisional transaction, the loss of wealth of one agent is precisely the gain of the other. We specify the important properties that the models considered herein satisfy via three structural assumptions.

Let $k=0,1,2,\ldots;$ $\Delta t \in (0,1)$; and $\mathcal{Z}$ be a mean-zero random variable with outcomes in $[-1,1]$.\footnote{Different models, within the class described, have different centered, random perturbations, $\mathcal{Z}$; for example, $\mathcal{Z}_1 \sim \frac{1}{2}(\delta_{-1}+\delta_{+1})$ and $\mathcal{Z}_2 \sim \text{Unif}([-1,1])$ correspond to two different models.} Two agents with indices $i$ and $j$ are selected at random at time $k\Delta t$. Denote the wealth of each agent by $w^i_{k \Delta t}$ and $w^j_{k\Delta t}$ respectively. The binary transaction between the agents indexed by $i$ and $j$ is 
\begin{equation}\label{eq:microtransaction}
\begin{pmatrix}
w_{(k+1)\Delta t}^i\\
w_{(k+1)\Delta t}^j
\end{pmatrix} 
= 
\begin{pmatrix}
w_{k\Delta t}^i\\
w_{k \Delta t}^j
\end{pmatrix} 
+ \sqrt{\gamma\Delta t} \left(w^i_{k \Delta t} \minop w^j_{k \Delta t}\right)\mathcal{Z} 
\begin{pmatrix}
1\\
-1
\end{pmatrix},
\end{equation} where $\gamma \in (0,1)$ and $\minop$ is the $\min$ operator. Clearly \[ \mathbb{E}\left[\left(w^i_{k \Delta t} \minop w^j_{k \Delta t}\right)\mathcal{Z}\right] = 0,\] and define \[ \phi(w^i_{k \Delta t}, w^j_{k \Delta t}):=\mathbb{E}\left[\left(\left(w^i_{k \Delta t} \minop w^j_{k \Delta t}\right)\mathcal{Z}\right)^2\right].\] Note that $\gamma\Delta t \phi(w^i_{k \Delta t}, w^j_{k \Delta t}) = \gamma\Delta t\left(w^i_{k \Delta t} \minop w^j_{k \Delta t}\right)^2 \text{Var}(\mathcal{Z})$ is the component-wise variance of the transaction, \cref{eq:microtransaction}. The boundedness of the perturbation $\mathcal{Z}$ and the scaling by $\left(w^i_{k \Delta t} \minop w^j_{k \Delta t}\right)$ is such that no agent is ever sent to nonpositive wealth so long as all agents are initialized to positive wealth.

The system described above satisfies three structural assumptions, which shall be assumed to apply to all models in the remainder of the paper.

\begin{structAssumpt}[Conservation of wealth]\label{def:wealthConserving}
    The kinetic asset exchange models are wealth conserving, that is the total wealth of the initial condition is the total wealth for all $t>0.$
\end{structAssumpt}

\begin{structAssumpt}[Preservation of positivity]\label{def:positivityPreserving}
    The kinetic asset exchange models are positivity preserving: an initial condition in which each agent has positive wealth evolves so that each agent always has positive wealth.
\end{structAssumpt}

\begin{structAssumpt}[Unbiased exchange]\label{def:unbiased}
The kinetic asset exchange models are unbiased in the sense that \[\mathbb{E}\left[\begin{pmatrix}
w_{(k+1)\Delta t}^i\\
w_{(k+1)\Delta t}^j
\end{pmatrix} 
-
\begin{pmatrix}
w_{k\Delta t}^i\\
w_{k \Delta t}^j
\end{pmatrix} \right]=\begin{pmatrix}
0\\
0
\end{pmatrix}.\]
\end{structAssumpt}

In the language of stochastic processes, an unbiased exchange is a martingale.

The diffusion approximation to the system given by \cref{eq:microtransaction} as $N\rightarrow\infty$ and $\Delta t \rightarrow 0$ is the second-order integro-differential equation of McKean-Vlasov type \cite{MR3443169} \begin{equation}\label{eq:wealthEvolutionEquation}
\ddt{\rho_t(w)} = \twoddw{}\left[\frac{\gamma}{2}\left(\int_{\mathbb{R}^+}\,dy\, \phi(w,y)\rho_t(y)\right)\rho_t(w)\right],\end{equation} where $\rho_t(w)$ is the density of economic agents at wealth $w$ at time $t$. Similar diffusion approximations to idealized economic systems have been studied in \cite{MR3623598, MR3872473}. The absence of a first-order drift term is attributable to the fact that these systems are martingales. The coefficient of the second-order term has the standard mean-field interpretation, as we explain below.

Define the diffusion coefficient \begin{equation}\label{eq:wealthDiffusionCoeff}\diff{w}{\rho} = \frac{\gamma}{2}\int_{\mathbb{R}^+}\,dy\, \phi(w,y)\rho(y),\end{equation} which is nonnegative, and allows \cref{eq:wealthEvolutionEquation} to be expressed as \begin{equation}\label{eq:wealthEvoEqSimple}
\ddt{\rho_t(w)} = \twoddw{}\bigg(\diff{w}{\rho_t}\rho_t(w)\bigg).\end{equation} 

\begin{remark}
    The diffusion coefficient \cref{eq:wealthDiffusionCoeff} is proportional to the time-infinitesimal expected variance experienced by an agent interacting with the mean-field population.
\end{remark}

Evolution equations in this class have at least two constants of motion: (i) the zeroth moment of $\rho$ corresponding to conservation of probability mass and (ii) the first moment of $\rho$ corresponding to conservation of wealth in the closed economic system. In what follows, we always take $\int_{\mathbb{R}^+}\,dw\,\rho_t(w) = 1 = \int_{\mathbb{R}^+}\,dw\,w\rho_t(w)$ for all $t>0$. When these conditions are imposed for the initial condition at $t=0$, the dynamics conserve them.

The natural boundary conditions are $\rho_t(0) = 0$ and $\lim\limits_{w\rightarrow \infty} \rho_t(w) = 0$ for all $t\geq 0.$

\subsection{A thermodynamic second law for kinetic asset exchange}\label{ssec:kaemSecondLaw}

A growing body of work in the field of econophysics suggests there is a principle for kinetic asset exchange models that mirrors the Second Law of Thermodynamics: In an unbiased economic system ``without regularization,'' measures of economic inequality will increase. This broad statement is supported by the many papers, see for example \cite{MR4686618,MR3428664,borgers2023new,MR4544046,MR4267576,MR3475485,MR4185105}, with proofs of the monotonicity of inequality under different kinetic asset exchange systems. We will now make precise the general principle.

Let $M_1:=\{\mu \in  \pac{\posreal} \,:\, \int_{\mathbb{R}^+}\,d\mu(x)\,x = 1\}$ be the subset of probability measures with unit first moment. When the initial condition to a kinetic asset exchange belongs to $M_1$ then the dynamics occur in $M_1$.

Define $G: M_1 \rightarrow \left[-\frac{1}{2},0\right)$ by \begin{equation}\label{eq:scaledGini}
G[\rho] = -\frac{1}{2}\int_{\mathbb{R}^+}\,dx\,\int_{\mathbb{R}^+}\,dy (x\minop y)\rho(x)\rho(y).
\end{equation}

\begin{remark}
The Gini coefficient, $\gini$, is a measure of economic inequality \cite{MR3012052} and when restricted to wealth distributions in $M_1$ has value $\giniarg{\rho} = 2G[\rho]+1$. Moreover, any wealth distribution in which the population has non-negative wealth can be canonically mapped into $M_1$ without changing its Gini coefficient.
\end{remark}

We will refer to $G$ as the \textit{scaled Gini coefficient}. 

\begin{lemma}\label{lem:GiniIden}
    The second-order derivative of the Fr\'echet derivative of $G$ is $\rho(x)$.
\end{lemma}

We include the proof despite its simplicity since the result is critically important in the study of kinetic asset exchange models.

\begin{proof}
Since the integral kernel of $G$ is symmetric in its arguments, the Fr\'echet derivative of $G$ is \[
\frechet{G}{\rho(x)} = -\int_{\mathbb{R}^+}\,dy\, (x\minop y)\rho(y).
\] The first $x$-derivative of the Fr\'echet derivative is
\[
\totdd{}{x}\frechet{G}{\rho(x)} = -\int_x^\infty \,dy\,\rho(y)
\] and the next $x$-derivative is
\[
\tottwodd{}{x}\frechet{G}{\rho(x)} = \rho(x).
\]
\end{proof}

Inserting this expression into dynamics of \cref{eq:wealthEvoEqSimple} yields \begin{equation}\ddt{\rho_t(w)} = \twoddw{}\left[\diff{w}{\rho_t}\twoddw{}\frechet{G[\rho_t]}{\rho_t}\right]\label{eq:kaemDiffGini}.\end{equation}

\begin{proposition}\label{prop:giniMonotoneEconophysics}
    The Gini coefficient of economic inequality is monotone increasing under the dynamics of the kinetic asset exchange models defined above.
\end{proposition}

By this we mean that if $\rho_t$ is a solution to the evolution equation \cref{eq:wealthEvolutionEquation} then $t\mapsto \giniarg{\rho_t}$ is an increasing function.

\begin{proof}
    The Gini coefficient $\gini$ (over $M_1$) is equal to $2G+1$ so if $\dot{G}\geq 0$ then so too for the Gini coefficient. By direct calculation, \begin{align*}
    \totdd{}{t}\left(G\circ \rho_t\right) &= \int_{\mathbb{R}^+} \,dw\, \frechet{G[\rho_t]}{\rho_t}\ddt{\rho_t} &&\text{definition of the first $L^2$ variation}\\
    &= \int_{\mathbb{R}^+} \,dw\, \frechet{G[\rho_t]}{\rho_t}\twoddw{}\left[D[w,\rho_t]\twoddw{}\frechet{G[\rho_t]}{\rho_t}\right] &&\text{via \cref{eq:kaemDiffGini}} \\
    &= \int_{\mathbb{R}^+} \,dw\, D[w,\rho_t]\left(\twoddw{}\frechet{G[\rho_t]}{\rho_t}\right)^2 &&\text{two IbP and imposing boundary conditions}\\
    &\geq 0 &&\text{since $D[w,\rho]\geq 0$.}
    \end{align*}
\end{proof}

This result can be strengthened to strict monotonicity under the assumption that the variance of the stochastic exchange only vanishes if an agent has zero wealth. See \cite{MR4804189}.

Inequality inexorably increases under the mean-field dynamics of kinetic asset exchange models satisfying \cref{def:wealthConserving,def:positivityPreserving,def:unbiased}. We can now formulate a so-called ``second law of econophysics.''

\begin{principle*}[Second Law of Thermodyanmics]
In an isolated physical system, Boltzmann entropy does not decrease.    
\end{principle*}

\begin{principle}
In an unbiased, wealth-conserving, positivity-preserving kinetic asset exchange, the Gini coefficient of economic inequality increases.
\end{principle}

\begin{principle*}[Second Law of Thermodynamics, alternative statement]
In an isolated physical system, heat flows from hot to cold and work must be performed to reverse the direction of the flow.
\end{principle*}

\begin{principle}
In an unbiased, wealth-conserving, positivity-preserving kinetic asset exchange, wealth flows from poor to rich and exogenous work would be necessary to reverse the direction of the flow.
\end{principle}

Given a Lyapunov functional for the entire class of systems, the natural question is: Are the evolutionary dynamics driven by the monotone functional as a gradient flow over an infinite-dimensional space? In other words, are these systems evolving to maximally increase economic inequality at each instant, in the way that $2$-Wasserstein theory shows for the Boltzmann entropy and the heat equation?

\subsection{An incompatibility result with the \texorpdfstring{$W_2$}{W2} gradient structure}\label{ssec:ysmNoGo}

The classical linear mobility 2-Wasserstein theory cannot describe, as gradient flows, some kinetic asset exchange models. While the nonlinear mobility theory could possibly handle such models, we argue in \cref{apdx:YSMW2mDiscussion} that it is not intuitively grounded for these particular models.

We show that there does not exist a $C^1$ functional whose $W_2$ gradient flow realizes a given evolution equation. The primary tool used to investigate the existence of such functionals is, in essence, a generalization of the potential condition for vector fields as a result of the fundamental theorem of calculus. See \cite{MR4331350, MR3014456, Vainberg1964}.

For a concrete example of a nonexistence result, consider the \textit{yard-sale model} \cite{MR3428664}, which is the simplest, fair transaction that does not send agents to negative wealth. In this model the random perturbation $\mathcal{Z}$ assigns equal probability mass to the outcomes $-1$ and $+1$, in which case $\phi(w,y) = (w\minop y)^2$. The evolution equation is \begin{equation}\label{eq:ysmPDE}
\ddt{\rho_t(w)} = \twoddw{}\left[\frac{\gamma}{2}\left(\int_{\mathbb{R}^+}\,dy\, (w\minop y)^2\rho_t(y)\right)\rho_t(w)\right].\end{equation}

\begin{proposition}\label{prop:nosmoothfuncYSM}
Let $\Omega = \mathbb{R}^+$. There is no $C^1$ $F:\ptwoac \rightarrow \mathbb{R}$ such that the yard-sale model of \cref{eq:ysmPDE} is realized as the gradient system $\left(\ptwoac, \mathbb{K}_\text{Otto}, F\right)$.
\end{proposition} The proof is given in \cref{apdx:noGoProof}. 


Since \cref{prop:nosmoothfuncYSM} shows that the classical $W_2$ theory is insufficient at least for the yard-sale model then it is reasonable to think that the next attempt at a gradient flow formulation should be with the nonlinear mobility theory. \Cref{apdx:YSMW2mDiscussion} discusses our reasoning for not dwelling on $W_{2,m}$ while in search of a sufficient metric geometry and serves as a guide to the thought process that led to the differential structure presented in \cref{sec:novelMetric}. The arguments therein about conserved quantities and extremization of functionals make use of several simplifying assumptions. Therefore we do not preclude other gradient flow formulations of phenomena with conserved quantities, see for example \cite{MR4746872, MR4179806}, rather the discussion gives a hint at the sufficiency of the structure while proffering an intuitive explanation about why $2-$Wasserstein geometry did not work, in this case.

We suggest that a fourth-order Onsager operator could sufficiently accommodate the dynamics that we seek to pose as gradient systems, rather than the second-order Otto Onsager operator, $\kotto$. 

\section{An adapted Wasserstein metric tensor}\label{sec:novelMetric}

We now describe the formal metric structure on an infinite-dimensional space that enables the map \begin{equation}\label{eq:w22dGradMap}\frechet{F}{\rho} \mapsto \twoddx{}\left[\varfunc{x}{\rho}{H}\twoddx{}\frechet{F}{\rho}\right],\end{equation} to be realized as given by the associated Onsager operator. This metric structure will allow the kinetic asset exchange models of \cref{sec:KAEM} to be expressed as gradient flows of the scaled Gini functional.

Let $\Omega \subset \mathbb{R}^n$ be a bounded, convex domain with smooth boundary, $\partial \Omega$, and \begin{equation}\label{eq:diffInw22D}\varfunc{x}{\rho}{D} = \int_{\Omega}\,dy\, W(x,y)\rho(y),\end{equation} where $W:\Omega \times \Omega \rightarrow \mathbb{R}^+$ is symmetric and $\mathcal{L}^n$-a.e. positive. Let $T_\rho \cdspace \subset \left\{ h\in\ltwo{}(\Omega)\,:\, h \perp_{L_2} \text{ker}(\Delta)\right\}$; this choice of tangent space is explored more in \cref{sec:idenDiffCDspaces}.

\begin{definition}\label{def:cdMetricTensor}
    At each $\rho\in\pac{\Omega}$, the metric tensor $\langle\,,\,\rangle_{\rho,\cdspace}:T_\rho \cdspace \times T_\rho \cdspace \rightarrow \mathbb{R}$ is defined for $h_1,h_2 \in T_\rho \cdspace$ as \begin{equation}\label{eq:w22Dtensor}
    \begin{split}
\langle h_1,h_2\rangle_{\rho,\cdspace} = &\int_\Omega\,dx\,\varfunc{x}{\rho}{D}\Delta\psi_1\Delta\psi_2\\
&\text{where for } i =1,2 \, \begin{cases}
\Delta(\varfunc{x}{\rho}{D}\Delta\psi_i)=h_i &\text{in } \Omega\\
\Delta \psi_i = 0 = \ddNorm{\Delta\psi_i}  &\text{on } \partial\Omega.
\end{cases}
\end{split}
\end{equation}
\end{definition}

Whereas the the metric tensor of $W_2$ involves second-order elliptic differential equations, the defining PDE for $\langle\,,\rangle_{\rho,\cdspace}$ are of fourth order. The state-dependent $\varfunc{x}{\rho}{D}$ term plays a similar role to the nonlinear mobility in the Riemannian structure of $W_{2,m}$.

\begin{definition}[$\cdspace$ space]
    Let $\cdspace(\Omega) = \{\pac{\Omega}, \langle\,,\,\rangle_{\rho,\cdspace} \}$. 
\end{definition}

We now turn to identifying for the $\cdspace$ spaces what is called the Otto calculus for $W_2,$ namely the first-order functional calculus.

\begin{proposition}[The first-order calculus of $\cdspace(\Omega)$]\label{prop:cdspacecalc}
    For $H:\pac{\Omega} \rightarrow \mathbb{R}$, \begin{equation}\label{eq:w22dgrad}\grad{\cdspace}{H[\rho]} = \wgenMobgradN{H[\rho]}{\varfunc{x}{\rho}{D}}\end{equation} or in one dimension
\begin{equation}\label{eq:w22dgrad1d}\grad{\cdspace}{H[\rho]} = \wgenMobgrad{H[\rho]}{\varfunc{x}{\rho}{D}}{x}.\end{equation}
\end{proposition} 
\begin{proof}
Consider a sufficiently smooth functional $H:\pac{\Omega} \rightarrow \mathbb{R}$. Let \begin{equation*}
    \begin{cases}
\Delta(\varfunc{x}{\rho}{D}\Delta\psi_1)=\grad{\cdspace}{H[\rho_0]} &\text{in } \Omega\\
\Delta(\varfunc{x}{\rho}{D}\Delta\psi_2)=\ddt{\rho_t}\big|_{t=0} &\text{in } \Omega\\
\Delta\psi_i=0=\ddNorm{\Delta\psi_i} &\text{ for i = 1,2 on } \partial\Omega.
\end{cases}
\end{equation*}

For any smooth curve $\rho_t:(-\epsilon, +\epsilon)\rightarrow\pac{\Omega}$, we have that \begin{equation}\label{eq:w22Drateofchange}
    \left\langle\grad{\cdspace}{H[\rho_0]},\ddt{\rho_t}\bigg|_{t=0}\right\rangle_{\rho_0,\cdspace} = \totdd{H[\rho_t]}{t}\bigg|_{t=0} = \left\langle\frechet{H[\rho_0]}{\rho},\ddt{\rho_t}\bigg|_{t=0}\right\rangle_{L^2}.
\end{equation}

By applying Green's identities and the homogeneous Neumann boundary conditions on $\psi_i$, the l.h.s. of \cref{eq:w22Drateofchange} becomes \begin{align}
    \int_\Omega\,dx\,\varfunc{x}{\rho_0}{D} \Delta \psi_1 \Delta \psi_2 &= \int_\Omega\,dx\, \Delta\left(\varfunc{x}{\rho_0}{D} \Delta \psi_1\right)\psi_2 \nonumber \\
    &= \int_\Omega\,dx\, \grad{\cdspace}{H[\rho_0]}\psi_2 ,\label{eq:w22drateofchangelhs}
\end{align} whereas upon inserting the definition of $\psi_2$ into the r.h.s. of \cref{eq:w22Drateofchange}
\begin{align}
    \int_\Omega\,dx\,\frechet{H[\rho_0]}{\rho}\ddt{\rho_t}\bigg|_{t=0} &= \int_\Omega\,dx\,\frechet{H[\rho_0]}{\rho}\Delta \left(\varfunc{x}{\rho_0}{D}\Delta \psi_2\right) \nonumber \\
    &= \int_\Omega\,dx\, \Delta\left(\varfunc{x}{\rho_0}{D}\Delta \frechet{H[\rho_0]}{\rho}\right)\psi_2. \label{eq:w22drateofchangerhs}
\end{align}

Since \cref{eq:w22drateofchangelhs,eq:w22drateofchangerhs} are equal for all $\psi_2$, we have the conclusion that \begin{equation}\grad{\cdspace}{H[\rho]} = \wgenMobgradN{H[\rho]}{\varfunc{x}{\rho}{D}}.\end{equation}
\end{proof}

Thus the state-dependent Onsager operator is \begin{equation}
    \mk_{\cdspace}[\cdot] = \Delta\left(\varfunc{x}{\rho}{D}\Delta[\cdot]\right),
\end{equation} which allows that gradient of a functional $H$ to be expressed as \begin{equation}
    \grad{\cdspace}{H[\rho]} = \mk_{\cdspace}\circ\frechet{H}{\rho}.
\end{equation}

In \cref{sec:idenDiffCDspaces}, we discuss how solvability conditions for the elliptic PDE in the metric tensor of \cref{def:cdMetricTensor} modify the tangent space by incorporating constraints on the admissible directions of motion.

\subsection{Gradient flow formulation of kinetic asset exchange}\label{ssec:gradFlowKaem}

We now prove that kinetic asset exchange models satisfying  \cref{def:wealthConserving,def:positivityPreserving,def:unbiased} can be expressed as the gradient flow of the scaled Gini functional on the space of wealth distributions endowed with a metric tensor that depends on the specific binary transaction. The kinetics, which come entirely from the prescribed binary transaction of \cref{eq:microtransaction}, are incorporated into the Riemannian structure, $\langle\,,\,\rangle_{\rho,\cdspace},$ and are separated from the energetic functional, which is the scaled Gini functional, $G,$ \textit{across all of the models}.

Therefore the diffusion approximations of the class of models satisfying \cref{def:wealthConserving,def:positivityPreserving,def:unbiased} may all be viewed through a unified variational lens: Each evolves as a curve of steepest ascent of economic inequality in a space of wealth distributions whose geometry is determined by the specifics of the collisional microscopic transactions.

\begin{theorem}\label{thm:gradFlowEconophysics}
Consider a kinetic asset exchange model satisfying \cref{def:wealthConserving,def:positivityPreserving,def:unbiased}.
Define $\gamma \phi(w,y)$ to be the time-infinitesimal variance for a transaction between agents with wealths $(w,y)$. Let the diffusion coefficient $\varfunc{w}{\rho}{D} = (\gamma/2){\mathbb{E}}_{\rho(y)}[\phi(w,y)]$ be half the expected time-infinitesimal variance in the mean-field/diffusion approximation.  Then the evolution equation for the probability density of agents in wealth-space \[\ddt{\rho_t(w)} = \twoddw{}\left[\frac{\gamma}{2}\left(\int_{\mathbb{R}^+}\,dy\, \phi(w,y)\rho_t(y)\right)\rho_t(w)\right]\] is equivalent to the evolution generated by the $\cdspace$ gradient system \[\ddt{\rho_t(w)} = +\grad{\cdspace}{G[\rho_t]},\] where the $G$ is the scaled Gini coefficient, \[G[\rho] = -\frac{1}{2}\int_{\mathbb{R}^+}\,dw\,\int_{\mathbb{R}^+}\,dy (w\minop y)\rho(w)\rho(y).\]
\end{theorem}
\begin{proof}
Restrict to $M_1 = \{\mu \in \pac{\posreal}\,:\, \int d\mu \, x = 1.\}$ Using the expression for the diffusion coefficient, write the dynamics as \begin{align*}
\ddt{\rho_t} &= \twoddw{}\left[\varfunc{w}{\rho_t}{D}\rho_t\right] \\
&= \twoddw{}\left[\varfunc{w}{\rho_t}{D}\twoddw{}\frechet{G[\rho_t]}{\rho_t}\right] &&  \text{by \cref{lem:GiniIden}}\\
&= + \text{grad}_{\cdspace}G[\rho_t] && \text{by \cref{prop:cdspacecalc}} .
\end{align*}As we sought, the gradient system $(\cdspace, G)$ yields the time-evolution.
\end{proof}

Entropy was known to increase under the heat equation dynamics, but it was not until the seminal work of Jordan, Kinderlehrer and Otto \cite{MR1617171, MR1842429} using the novel $W_2$ metric structure that the heat equation could be viewed as a gradient flow of Boltzmann entropy.

Likewise, the Gini coefficient of economic inequality has been known to increase under many different econophysics models but it is only now with the metric structure of $\cdspace$, which has cleanly disentangled the kinetics and energetics, that the natural gradient structure of these econophysics models is evident. This is what we mean by an \textit{intuitive} gradient flow.

A rich thermodynamics interpretation is made available from the work of L. Onsager \cite{MR57765} and its interplay with Wasserstein theory \cite{MR2776123,MR3150642,mielke2023introduction}. Using this formulation, the following characterization is available for the evolutionary dynamics of wealth distributions:\begin{itemize}
    \item The (scaled) Gini coefficient of economic inequality is the free-energy functional of each system.
    \item The differential of $G$ with respect to the particle number (namely, the agent mass density), $DG \equiv \frechet{G}{\rho}$, is the chemical potential (the thermodynamical driving force), $\xi$.
    \item The Onsager operator, $\mk_{\cdspace}$ -- the only object that reflects the specific kinetics of a particular model -- couples the driving force, $\xi$, to the flux of agent mass in wealth space $\partial_t\rho_t$ by way of the gradient flow theorem.
\end{itemize}

\begin{remark}
    The evolution equations of \cref{thm:gradFlowEconophysics} preserve convex order; that is, $\totdd{}{t}\int_\posreal\,dw\,\phi(w)\rho_t(w) \geq 0$ for convex $\phi$. While this structure aligns with the martingale optimal transport formulation of \cite{MR4003563}, the present framework is designed to accommodate future work incorporating drift terms (for example, those in \cite{MR3872473}) that may mitigate the monotone increase of economic inequality. Since these terms typically disrupt the preservation of convex order, such a system moves beyond the strict scope of martingale optimal transport.
\end{remark}

\subsection{Conserved quantities of \texorpdfstring{$\cdspace$}{C D} gradient flows}\label{ssec:cdGradFlowConsrvQuant}

Gradient flows in $\cdspace$ naturally incorporate conserved quantities. \Cref{apdx:YSMW2mDiscussion} discusses how the conserved quantities for the wealth distribution evolution equations motivated the $\cdspace$ metric tensor.


\begin{lemma}\label{lem:ConservedQuant}
    Let $F: \pac{\Omega} \rightarrow \mathbb{R}$ be a smooth functional for which \begin{equation}
        \varfunc{x}{\rho}{D}\Delta\frechet{F}{\rho} \quad \text{and} \quad \ddNorm{}\left[\varfunc{x}{\rho}{D}\Delta\frechet{F}{\rho}\right] \nonumber
    \end{equation} vanish on $\partial\Omega$, where $\ddNorm{}$ is the normal derivative on the boundary. If $\eta: \Omega \rightarrow \mathbb{R}$ is harmonic in $\Omega$ then \[\int_\Omega \,dx\, \eta(x)\rho_t(x)\] is constant in time under the dynamics of the gradient flow \[\ddt{\rho_t} = \text{grad}_{\cdspace}F[\rho_t].\]
\end{lemma}

\begin{proof}
    Simply note that \[\totdd{}{t} \int_\Omega\,dx\,\eta(x)\rho_t(x) = \int_\Omega\,dx\,\eta(x) \left(\Delta\left(\varfunc{x}{\rho_t}{D}\Delta \frechet{F}{\rho}\right)\right)\] and integrate by parts.
\end{proof}

\begin{remark}[Dissipative invariants]
    In a Hamiltonian system with Poisson bracket $\{\,,\,\}$, observable $a$, and Hamiltonian $H$, the evolution is generated by $\dot{a} = \{a,H\}$. The \textit{Casimir invariants} of a Poisson bracket are observables whose bracket with any observable vanish and are therefore constant in time \cite{PhysRevA.33.4205}.
    
    In \cite{MR1627532}, P.J.~Morrison states Casimir invariants ``are built into the phase space and are in this sense kinematic constants of motion.''  So too are the \textit{dissipative invariants} (a term introduced in \cite{MR735549}) of \cref{lem:ConservedQuant}, which are a consequence of the geometry imposed on the space by the metric tensor.

    The gradient structure above is a dissipative bracket $[\, , \, ]$ in the sense of \cite{MR735549} under the identification \begin{equation}
        \grad{\cdspace}{F[\rho_t]} = \left[\rho_t,F\right]_{\cdspace}\nonumber.
    \end{equation}

    The notion of a dissipative Hamiltonian system with both a Poisson and dissipative bracket along with a Hamiltonian and free-energy functional has been codified into the GENERIC formulation \cite{HCOthermodynamics}.
\end{remark}

We will return to discussing these conservation laws in \cref{sec:idenDiffCDspaces}.

\section{The case of \texorpdfstring{$\varfunc{x}{\rho}{D} = \rho$}{D[x,p]=p}}\label{sec:idenDiffCDspaces}

When $\varfunc{x}{\rho}{D} = \rho$, the corresponding metric structure mirrors that of linear mobility, i.e. $m(\rho) = \rho,$ in the classical $W_2$ theory. At each $\rho\in\pac{\Omega}$ in $\cdspacearg{\rho}$, the metric tensor $\langle\,,\,\rangle_{\rho,\cdspacearg{\rho}}:T_\rho\cdspacearg{\rho}\times T_\rho \cdspacearg{\rho} \rightarrow \mathbb{R}$ is defined for $h_1,h_2 \in T_\rho W_\cdspacearg{\rho}$ as \begin{equation}\label{eq:w22tensor}
\begin{split}
\langle h_1,h_2\rangle_{\rho,\cdspacearg{\rho}} = &\int_\Omega\,dx\,\rho\Delta\psi_1\Delta\psi_2\\
&\text{where for } i =1,2 \, \begin{cases}
\Delta(\rho\Delta\psi_i)=h_i &\text{in } \Omega\\
\Delta \psi_i = 0 = \ddNorm{\Delta\psi_i}  &\text{on } \partial\Omega.
\end{cases}
\end{split}
\end{equation} For the $\psi_i$ coupled to the $h_i$ via the $\rho$-weighted biharmonic equation with homogeneous Neumann boundary conditions, we have that\footnote{Recall our convention that $\mu = \rho \mathcal{L}^n.$} \[\langle h_1,h_2\rangle_{\rho,\cdspacearg{\rho}} = \langle \Delta\psi_1,\Delta\psi_2\rangle_{\ltwo{\mu}}.\]

In this section, we investigate the consequences of expressing the metric as a time-integrated $\rho$-dependent version of a dual norm on a space that arises in the study of the $\rho$-weighted biharmonic equations with homogeneous Neumann boundary conditions. Here we essentially study the adapted Benamou-Brenier formula.

We mention how first-order negative Sobolev norms appear in $W_2$ then introduce the spaces $\htildetwo{\mu}$ and $\htildetwoneg{\mu}$ and express the metric of $\cdspacearg{\rho}$ in terms of the latter's norm. 

The one-dimensional case is rigorously analyzed in \cite{MR4961432,MR4987462}, including an application to the Derrida-Lebowitz- Speer-Spohn (DLSS) equation, and discussed as a weakening of the dynamic formulation of the \textit{martingale optimal transport problem} from \cite{MR4003563}.

\subsection{\texorpdfstring{$W_2$}{W2} and negative homogeneous Sobolev norms}
The $W_2(\Omega)$ metric has a number of connections to negative, first-order homogeneous Sobolev norms. In particular, the metric can be expressed as the action minimizing \[W_2^2(\mu,\nu) = \inf_{\text{curves }\mu_t}\left\{\int_0^1\,dt\,||\dot{\mu}_t||_{\dot{H}^{-1}_{\mu_t}}^2\,:\,\mu_0 = \mu, \mu_1 = \nu\right\}\] or equivalently the length minimizing\[W_2(\mu,\nu) = \inf_{\text{curves }\mu_t}\left\{\int_0^1\,dt\,||\dot{\mu}_t||_{\dot{H}^{-1}_{\mu_t}}\,:\,\mu_0 = \mu, \mu_1 = \nu\right\},\] where \[||f||_{\dot{H}^{-1}_{\mu}} := \int_\Omega\,d\mu\,|\nabla \psi_f|^2\] and $\psi_f$ is defined by the elliptic equation \begin{equation}\label{eq:muWeightedPoisson}\begin{cases}
\divg \left(\rho \nabla \psi_f\right) = f &\text{in } \Omega\\
\ddNorm{\psi_f} = 0 &\text{ on } \partial\Omega.
\end{cases}\end{equation}

The first-order weighted homogeneous Sobolev spaces $\hdotone{\mu}(\Omega)$ and $\hdotoneneg{\mu}(\Omega)$ arise naturally in the study of \cref{eq:muWeightedPoisson}, a $\mu$-weighted Poisson problem with Neumann boundary conditions. A classical result is that there exists a solution\footnote{Any solution, $u(x)$, is not unique as any scalar offset, $u(x)+c$, still satisfies the equation.} to \cref{eq:muWeightedPoisson} for $f\in L^2$ if and only if \[\int_\Omega\,dx\, f = 0.\] Of course the mean-zero condition holds for the data $\dot{\mu_t}$ since $\mu_t$ is a probability measure for all $t\in (0,1)$. The mean-zero condition on $f$ can be rephrased as $f\perp_{L^2} \ker(\nabla)$. The space $\dot{H}^1_\mu(\Omega)$, in which uniqueness to \cref{eq:muWeightedPoisson} is best studied, is defined as the factor space of $H^1_\mu(\Omega) = \left\{u \in L^2_\mu(\Omega) \,:\, D^\alpha u \in L^2_\mu(\Omega) \,\forall\, |\alpha| = 1\right\}$ under the equivalence relation $u\sim v$ if $u-v \in \ker(\nabla).$

A number of transport inequalities \cite{MR2895086} and asymptotic comparisons make use of the $\hdotone{\mu}$ and $\hdotoneneg{\mu}$ structure. For example, see \cite{MR3922440}, \cite[Section~5.5.2]{MR3409718}, \cite[Theorem~7.2.6]{MR1964483}, and \cite[Exercise~22.20]{MR2459454}.

\subsection{\texorpdfstring{$\mu$}{u}-weighted biharmonic equation with Neumann boundary conditions}
We turn our attention to the elliptic equations for $\psi_i$ in \cref{eq:w22tensor}. 

Let $\Omega\subset \mathbb{R}^n$ be an open, convex, bounded domain with smooth boundary (assume that $\partial \Omega$ is at least $C^4$). Consider the fourth-order PDE \begin{equation}\label{eq:biharmonicNeumann}\begin{cases} \Delta \left(\rho \Delta u\right) = f \quad \text{in } \Omega \\ \frac{\partial \Delta u}{\partial \nu} = 0 = \Delta u \quad \text{on } \partial\Omega, \end{cases}\end{equation} where $\rho:\Omega\rightarrow \mathbb{R}^+$ is the density of a positive Borel measure on $\Omega$ that is absolutely continuous with respect to the Lebesgue measure on $\Omega$ and $f\in L^2(\Omega)$. We seek to understand the existence and uniqueness of solutions to this fourth-order PDE with Neumann boundary conditions.

The correct space to study the biharmonic equation with homogeneous Dirichlet boundary conditions is the trace-zero, second-order Sobolev space $H^2_0(\Omega)$ for which the existence and uniqueness results are easily seen. In fact the dual norm for $H^2_0(\Omega)$ looks promising, i.e. an inner product structure like $(\Delta u, \Delta v)_{L^2}$, except that it encodes the wrong boundary conditions. We make use of results from \cite{valli2023wellposed} with reference to \cite{MR2597943, MR2667016, MR2777530}

The $\mu$-weighted space $L^2_\mu(\Omega)$ has norm \[||f||_{\ltwo{\mu}}^2=\int_\Omega\,d\mu(x)\, f(x)^2 .\]Let $L^2_\mu(\Delta;\Omega) = \{u \in L^2_\mu(\Omega)\,:\,\Delta u \in L^2_\mu(\Omega)\}$ and $\mathcal{H}_\mu = \{h \in L^2_\mu(\Omega)\,:\, \Delta h = 0\}$. Of course, $\mathcal{H}_\mu \subset \ltwomudelta$. The differential operator $\Delta$ should be viewed in the distributional or weak sense.

We say that $u\in L^2_\mu(\Delta;\Omega)$ is a weak solution to \cref{eq:biharmonicNeumann} if \begin{equation}\label{eq:weakBiharmNeumann}\int_\Omega\,d\mu(x)\,\Delta u \Delta v = \int_\Omega \,dx\, fv\end{equation} for all $v\in L^2_\mu(\Delta;\Omega)$. The existence of weak solutions to the $\mu$-weighted biharmonic equation with Neumann boundary conditions depends on the conditions imposed on $f$.

Any harmonic function, $h\in\mathcal{H}_\mu$, can be added to a weak solution $u$ of the Neumann biharmonic problem and the result $u+h$ remains a solution since $\mathcal{H}_\mu \subseteq \ker(\Delta).$ and $\mathcal{H}_\mu \subset L^2_\mu(\Delta;\Omega)$. Thus, if we are to find a space in which the Lax-Milgram theorem or the Riesz Representation Theorem may be applied, then non-zero harmonic functions must be excluded or collapsed to zero by a quotient construction, since either of those two theorems yields a uniqueness result.

The first conditions imposed on $f$ are that $f\in L^2(\Omega)$ and $f \perp_{L^2} \mathcal{H}_\mu$, that is, for all $h \in \mathcal{H}_\mu$, $\int_\Omega\,dx\, fh = 0.$ This condition is natural given the previous comment about the uniqueness: Let $u\in\ltwomudelta$ be a weak solution, $v\in \ltwomudelta$ be arbitrary, and $h\in\mathcal{H}_\mu$ then \begin{align}
0 &= \int_\Omega \,d\mu(x)\, \Delta(u)\Delta\left(v+h\right) - \int_\Omega\,d\mu(x)\,\Delta(u)\Delta(v) \nonumber \\
&= \int_\Omega\,dx\, f \left(v+h\right)- \int_\Omega\,dx\, f(v) \nonumber \\
&= \int_\Omega \,dx\, fh.\label{eq:meanZeroConditionF}
\end{align}

\begin{proposition}
    For each $\mu$ with density bounded below by a positive constant for a.e. $x\in\Omega$, the weak form \cref{eq:weakBiharmNeumann} of the $\mu$-weighted biharmonic equation \cref{eq:biharmonicNeumann} with $f\in L^2(\Omega)$ has a solution if and only if $f \perp_{L^2} \mathcal{H}_\mu$.
\end{proposition}

If a weak solution $u \in \ltwomudelta$ exists, then $f\perp_{\ltwo{}} \mathcal{H}_\mu$. Since $\mathcal{H}_\mu \subset \ltwomudelta$, choose $v\in \mathcal{H}_\mu$ in the definition of the weak solution to see that $\int_\Omega \,dx\, f v = 0$. 

As for the converse, assume that $f\perp_{L^2}\mathcal{H}_\mu$. Let $\htildetwo{\mu} = \ltwomudelta / \mathcal{H}_\mu$, the quotient space defined by $u\sim v$ if $u-v\in\mathcal{H}_\mu$. For two equivalence classes $[u], [v] \in \htildetwo{\mu}$, define the $\htildetwo{\mu}$ inner product to be $(\Delta [u], \Delta [v])_{\ltwo{\mu}}$, which is single valued over all members of the equivalence classes. The l.h.s. of \cref{eq:weakBiharmNeumann} is exactly the inner product on $\htildetwo{\mu}$. To make use of the Riesz Representation Theorem, we must show that $(f,\cdot)_{\ltwo{    }} \in \text{Dual}\left(\htildetwo{\mu}\right)=:\htildetwoneg{\mu}$.

We show boundedness -- that is, there exists $C>0$ such that $(f,[u])_{\ltwo{}}\leq C ||u||_{\htildetwo{}}$ --  for the Lebesgue case. The assumption gives that $(f,[u])_{\ltwo{}}$ is single valued over an equivalence class. That is, $(f,u-h)_{\ltwo{}}$ is constant for all $h \in \mathcal{H}$. Therefore by the Cauchy-Schwarz inequality, \[(f,[u])_{\ltwo{}} \leq ||f||_{\ltwo{}}\inf_{h\in\mathcal{H}}||u-h||_{\ltwo{}},\] for some arbitrary $u\in [u]$, and by the Hilbert projection theorem \[(f,[u])_{\ltwo{}} \leq ||f||_{\ltwo{}}||u-P_{\mathcal{H}}u||_{\ltwo{}},\] where $P_{\mathcal{H}}$ is the orthogonal projection onto harmonic functions.

If $[u] = [0]$ then $(f,[u])_{\ltwo{}} = 0$ by assumption so consider only the case $[u] \neq [0]$ so that $\hat{u} : = (u-P_{\mathcal{H}}u) \neq 0.$ Set $g = \Delta \hat{u}$ and note that $g\in \ltwo{}$ and $g \neq 0.$ Assuming that $\partial \Omega$ is $C^4$, then regularity results for the Poisson problem \cite{MR2667016} guarantee there exists $C>0$ such that \[||\hat{u}||_{\ltwo{}} \leq C||g||_{\ltwo{}}.\] Moreover $C$ depends only on the geometry of the domain and dimension of the problem and not on the $\hat{u}$ and $g:=\Delta\hat{u}.$ 

Making use of this regularity result and that $\Delta\hat{u} = \Delta u$, we have \[(f,[u])_{\ltwo{}} \leq C ||f||_{\ltwo{}} ||\Delta u||_{\ltwo{}} = \mathcal{C} ||[u]||_{\htildetwo{}}\] thus proving boundedness of the well-posed linear functional for the Lebesgue case.

At this point the Riesz representation theorem can be invoked to show that a solution to the original problem exists in the Lebesgue case.

Now, take $\mu$ to be a positive measure, not necessarily Lebesgue, for which there exists $\alpha>0$ such that $\rho(x)\geq \alpha$ for a.e. $x\in\Omega$. If $\mu$ is absolutely continuous with respect to the Lebesgue measure then $\mathcal{H} \subset \mathcal{H}_\mu.$ The proof above can be carried out to the point \[(f,[u])_{\ltwo{}} \leq C ||f||_{\ltwo{}}||\Delta u||_{\ltwo{}}\] and apply that \[||[u]||_{\htildetwo{dx}}\leq \frac{1}{\sqrt{\alpha}} ||[u]||_{\htildetwo{\mu}}.\] In which case, we can conclude \[(f,[u])_{\ltwo{}}\leq \mathcal{C}' ||[u]||_{\htildetwo{\mu}}.\]

\subsection{The spaces \texorpdfstring{$\htildetwo{\mu}$}{H2u} and \texorpdfstring{$\htildetwoneg{\mu}$}{H(-2)u}}

We now drop the equivalence class notation, $[u]$, for elements of $\htildetwo{\mu}$ in favor of just $u$. Let $\langle\,,\, \rangle$ denote the duality pairing. In the primal space $\htildetwo{\mu}$, the norm is \[||u||_{\htildetwo{\mu}}^2 = \int_\Omega\,d\mu(x)\, (\Delta u)^2.\]

There is an isometric isomorphism between $\htildetwo{\mu}$ and its dual $\htildetwoneg{\mu}$ given by \[\htildetwo{\mu}\ni u \mapsto \deltatwomu{\mu}u \in \htildetwoneg{\mu}\] where \[ \deltatwomu{\mu} u(v) := \int_\Omega\,d\mu(x)\, \Delta u \Delta v.\]

Since \[||\deltatwomu{\mu}u||_{\htildetwoneg{\mu}} = \sup\left\{|\deltatwomu{\mu}u(v)|\,:\,||v||_{\htildetwo{\mu}}\leq 1\right\}\] then \[||\deltatwomu{\mu}u||_{\htildetwoneg{\mu}}=||u||_{\htildetwo{\mu}}.\]

Fixing some $f\in \htildetwoneg{\mu}$, the convex functional \[P[u] = \int_\Omega\,d\mu(x)\, (\Delta u)^2 - 2\langle f,u \rangle\] has a unique minimizer. Therefore there exists $u_f \in\htildetwo{\mu}$ such that \[\langle f, v\rangle = \deltatwomu{\mu}u_f(v)\] for all $v\in\htildetwo{\mu}$. In particular $u_f$ is found by solving \[\Delta(\rho\Delta u_f) = f\] with the Neumann boundary conditions since \[\langle f,v \rangle = \deltatwomu{\mu}u_f(v) = (\Delta u_f, \Delta v)_{\ltwo{\mu}} = \langle \Delta(\rho\Delta u_f), v\rangle\] for all $v\in \htildetwo{\mu}.$

Let $f,g\in \htildetwoneg{\mu}$. Using the polarization identity, the inner product on $\htildetwoneg{\mu}$ is given by \begin{equation}\label{eq:htildetwonegIP}(f,g)_{\htildetwoneg{\mu}} = (\psi, \phi)_{\htildetwo{\mu}}= (\Delta \psi,\Delta \phi)_{\ltwo{\mu}} \text{ where }\begin{cases}
\deltatwomu{\mu} \psi =\Delta(\rho\Delta \psi) = f \\
\deltatwomu{\mu} \phi =\Delta(\rho\Delta \phi) = g \\
\psi, \phi \in \htildetwo{\mu}.
\end{cases}\end{equation}

The preceding relations between $f,g$ and $\psi, \phi$ allow the inner product to be expressed equivalently as \begin{align*}
(f,g)_{\htildetwoneg{\mu}} &= \int_\Omega\,dx\, f \left(\deltatwomu{\mu}\right)^{-1}g= \int_\Omega\,dx\, g\left(\deltatwomu{\mu}\right)^{-1}f\\
&= \int_\Omega\,dx\, f\phi = \int_\Omega\,dx\,g\psi.
\end{align*}

\begin{remark}
    There are similarities between the metric of \cref{eq:htildetwonegIP} and the Hessian metrics of $2-$Wasserstein space arising from transport information geometry studied in \cite{MR4222140}. In the former case, $\deltatwomu{\mu}$ becomes the isometric isomorphism between primal and dual spaces whereas in \cite{MR4222140}, the original operator $-\divg(\rho\nabla)$ between spaces is retained but the metric structure modified.
\end{remark}

\subsection{\texorpdfstring{$\cdspacearg{\rho}$}{Cp} using the dual norm}
With the inner product structure of $\htildetwoneg{\mu}$ established by \cref{eq:htildetwonegIP}, the metric of the adapted Wasserstein space $\cdspacearg{\rho}$ may be written using the new notation. The understanding of the effects on the tangent space show how the space of probability measures becomes foliated by the intrinsically conserved quantities discussed earlier in \cref{lem:ConservedQuant}, akin to the symplectic leaves of phase space in Hamiltonian dynamics.

Let $\nu = \sigma \mathcal{L}^n$ and $\lambda = \gamma \mathcal{L}^n$ be positive Borel measures that satisfy $(\sigma-\gamma) \perp_{\ltwo{}} \ker(\Delta)$. In analogy with the dynamic formulation of the $W_2$ metric, define \begin{equation}\label{eq:actionFormulationCRho}
    d_{\cdspacearg{\rho}}^2(\nu,\lambda) = \inf_{\mu_t}\left\{\int_0^1\,dt\,||\dot{\mu}_t||^2_{\htildetwoneg{\mu_t}} \,:\, \mu_0 = \nu, \mu_1 = \lambda\right\}
\end{equation} or equivalently \begin{equation}
    d_{\cdspacearg{\rho}}(\nu,\lambda) = \inf_{\mu_t}\left\{\int_0^1\,dt\,||\dot{\mu}_t||_{\htildetwoneg{\mu_t}} \,:\, \mu_0 = \nu, \mu_1 = \lambda\right\}.
\end{equation}

Close variants of \cref{eq:actionFormulationCRho} are studied as martingale optimal transport in \cite[eq.~(1.6)]{MR4003563} and the diffusive transport metric in \cite[eq.~(4)]{MR4987462}. In particular, \cite{MR4987462} focuses on the semi-contractivity properties of the metric and includes some explicit calculations with regard to the heat semigroup.

In $W_2$, the tangent element $\dot{\mu}_t\in \hdotone{\mu_t}$ necessarily satisfies $\dot{\mu}_t \perp \ker(\nabla)$, that is $\dot{\mu}_t$ is mean zero. In $\cdspacearg{\rho}$, the inclusion $\dot{\mu}_t \in \htildetwoneg{\mu_t}$ demands $\dot{\mu}_t \perp_{\ltwo{}} \ker(\Delta)$. That is, $h\in T_{\mu_t}\cdspacearg{\rho}(\Omega)$ necessitates \[ \Theta \in \mathcal{H}_\mu = \{\Theta \in \ltwo{\mu_t}(\Omega)\,:\, \Delta \Theta = 0\} \implies \int_\Omega\,dx\, h\Theta = 0,\] as a consequence of $h$ satisfying $\deltatwomu{\mu_t}\psi = h$ for some $\psi \in \htildetwo{\mu_t}.$

\begin{remark}[The consequence for kinetic asset exchange models]
If $\Omega = \posreal$ then the conserved quantities are expectation values of affine functions. Thus the infinitesimal action of a tangent space element when described \textit{by what it does not do} is precisely as sought: It changes neither the normalization (total probability) nor the first moment (total wealth).
\end{remark}

In higher spatial dimensions than the $n=1$ econophysics models considered in \cref{ssec:gradFlowKaem}, the conserved quantities given by the expectation values of $\Theta \in \mathcal{H}_\mu$ may have significant interest.

The expression \cref{eq:actionFormulationCRho} permits a statement about the moments of measures.

\begin{proposition}
    For $\Omega\subset \mathbb{R}^n,$ let $\nu,\lambda \in \pac{\Omega}$. If $d_{\cdspacearg{\rho}}(\nu,\lambda)<+\infty$ then each has finite fourth moment.
\end{proposition}

\begin{proof}
    For $\mu_t:[0,1]\rightarrow\pac{\Omega}$, let \begin{equation}
        \mathcal{E}_{\mu_t} = \int_0^1\,dt\,||\dot{\mu}_t||^2_{\htildetwoneg{\mu_t}}.
    \end{equation} By the finiteness of $d_{\cdspacearg{\rho}}(\nu,\lambda)$, there exists a curve $\mu_t$ such that $\mu_0 = \nu$, $\mu_1=\lambda$, and $\mathcal{E}_{\mu_t}<+\infty.$ Let $\rho_t\mathcal{L}^n = \mu_t$ and $\psi_t \in \htildetwo{\mu_t}$ solve \begin{equation}
        \deltatwomu{\mu_t} \psi_t = \dot{\rho}_t\nonumber
    \end{equation}

    We will produce a differential inequality on the fourth moment \begin{equation}
        m_4(t) = \int_\Omega\,dx\,\rho_t(x)|x|^4.\nonumber
    \end{equation} 

    Differentiating in time yields \begin{align}
        \left|\totdd{m_4(t)}{t}\right| &= \left|\int_\Omega\,dx\,\ddt{\rho_t}|x|^4\right|\nonumber \\ 
        &= 4(n+2)\left|\int_\Omega\,dx\, \rho_t(x)\left(\Delta\psi_t(x)\right)|x|^2\right| \nonumber \\
        &\leq 4(n+2)||\dot{\mu}_t||_{\htildetwoneg{\mu_t}}\left(m_4(t)\right)^{1/2} \nonumber,
    \end{align} where the final line invokes the Cauchy-Schwarz inequality.

    Therefore \begin{equation}
        \left| \totdd{}{t} \sqrt{m_4(t)}\right| \leq 2(n+2)||\dot{\mu}_t||_{\htildetwoneg{\mu_t}},
    \end{equation} so integrating in time gives
    \begin{equation}
        \left| \sqrt{m_4(r)} - \sqrt{m_4(s)}\right|\leq 2(n+2)\sqrt{\mathcal{E}_{\mu_t}} < +\infty,
    \end{equation}
    for $0\leq r,s \leq1$.
\end{proof}

By contrapositive, if either $\nu$ or $\lambda$ do not have finite fourth moment then $d_{\cdspacearg{\rho}}(\nu,\lambda) = +\infty$.

\subsection{Transport inequalities for \texorpdfstring{$\cdspacearg{\rho}$}{Cp}}

Here we develop several transport inequalities in $\cdspacearg{\rho}$ that make use of the $\htildetwoneg{\mu}$ structure. Each of these transport inequalities and infinitesimal results have been similarly established for $W_2$. Let $\pac{\Omega}$ be the set of positive probability measures of $\Omega$ that are absolutely continuous with respect to the Lebesgue measure, $\mathcal{L}^n$. For two positive measures $\lambda = \sigma \mathcal{L}^n$ and $\nu = \gamma \mathcal{L}^n$, $(\nu-\lambda)\in \htildetwoneg{dx}$ is equivalent to $(\gamma-\sigma) \perp_{\ltwo{}} \ker(\Delta).$

\begin{proposition}\label{prop:htildetwonegInequal1}
Let $\mu = \rho \mathcal{L}^n$ and $\mu' = \rho' \mathcal{L}^n$ be in $\pac{\Omega}$ such that $\mu' \geq \beta \mu$ for $\beta>0$ then \[||f||_{\htildetwoneg{\mu'}}\leq \beta^{-1/2}||f||_{\htildetwoneg{\mu}}\] for all $f \in \htildetwoneg{\mu}.$
\end{proposition}

\begin{proof}
Let \[S = \left\{\psi \in \htildetwo{\mu} \, : \, \left(\int_\Omega\,d\mu\,\left(\Delta\sqrt{\beta} \psi\right)^2 \right)^{1/2}\leq 1 \right\}\] and \[S' = \left\{\psi \in \htildetwo{\mu'} \, : \, \left(\int_\Omega\,d\mu'\,\left(\Delta \psi\right)^2 \right)^{1/2}\leq 1 \right\}.\]
Note that $S'\subset S$ since if $\psi \in S'$ then \begin{align*}
    1 &\geq \left(\int_\Omega\,d\mu'\,\left(\Delta \psi\right)^2 \right)^{1/2} \\
    &\geq \left(\int_\Omega\,d\mu\,\left(\Delta \sqrt{\beta}\psi\right)^2 \right)^{1/2}
\end{align*} by the assumption $\mu' \geq \beta \mu.$

Let $f\in\htildetwoneg{\mu}$ and note \begin{align*}
||f||_{\htildetwoneg{\mu'}} &= \sup_{\psi \in S'} \langle \psi, f \rangle \\
&\leq \sup_{\psi \in S} \langle \psi, f \rangle \\
&= \beta^{-1/2} \sup_{\psi \in S} \langle \beta^{1/2}\psi, f \rangle \\
&=  \beta^{-1/2}||f||_{\htildetwoneg{\mu}}.
\end{align*}
\end{proof}

A transport inequality follows easily from this result.

\begin{corollary} Let $\lambda,\nu \in \pac{\Omega}$ and $(\lambda-\nu)\in \htildetwoneg{dx}$ then
    \[d_{\cdspacearg{\rho}}(\lambda,\nu)\leq 2||\lambda-\nu||_{\htildetwoneg{\lambda}}.\]
\end{corollary}
\begin{proof}
  Since
\[d_{\cdspacearg{\rho}}(\lambda,\nu) = \inf_{\mu_t}\left\{\int_0^1\,dt\,||\dot{\mu_t}||_{\htildetwoneg{\mu_t}} \,:\, \mu_0 = \lambda, \mu_1 = \nu\right\}\] then choosing $\mu_t = (1-t)\lambda + t\nu$, \[d_{\cdspacearg{\rho}}(\lambda,\nu)\leq \int_0^1\,dt\,||\partial_t\left[(1-t)\lambda + t\nu\right]||_{\htildetwoneg{\mu_t}}\] and note that $\mu_t\geq (1-t)\lambda$. Apply \cref{prop:htildetwonegInequal1}.
\end{proof}

\begin{proposition}
Let $\lambda$ and $\nu$ be elements of $\pac{\Omega}$ whose densities are bounded from below by the same constant $0<\alpha<\infty$ and $(\nu-\lambda) \in \htildetwoneg{dx}$ then \[d_{\cdspacearg{\rho}}(\lambda,\nu)\leq \alpha^{-1/2}||\nu-\lambda||_{\htildetwoneg{dx}}.\]
\end{proposition}

\begin{proof}
Let $\mu_t$ be the linear interpolation between $\lambda$ and $\nu$ and note that $\rho_t \geq \alpha.$ Apply \cref{prop:htildetwonegInequal1}.
\end{proof}

\begin{proposition}
Let $\mu,\nu\in\pac{\Omega}$ be such that their densities are bounded above by $0<C<\infty$ and $(\nu-\mu) \in \htildetwoneg{dx}$. Suppose that the curve $\mu_t$ achieves the infimum in the action formulation, \cref{eq:actionFormulationCRho}, of the $\cdspacearg{\rho}$ metric and satisfies $0<\rho_t<C$ for $t\in (0,1)$. Then \[C^{-1/2}||\nu-\mu||_{\htildetwoneg{dx}}\leq d_{\cdspacearg{\rho}}(\mu,\nu).\]
\end{proposition}

\begin{proof}
Since $\mu_t$ is optimal, \[d_{\cdspacearg{\rho}}(\mu,\nu)^2 = \int_0^1\,dt\,||\dot{\mu}_t||_{\htildetwoneg{\mu_t}}^2.\] Let $\phi_{\dot{\mu}_t} \in \htildetwo{\mu_t}$ be such that for each $t\in [0,1]$ \[\deltatwomu{\mu_t}\phi_{\dot{\mu}_t}=\dot{\mu}_t.\] Recall the definition of the $\htildetwoneg{\mu_t}$ norm, \[||\dot{\mu}_t||_{\htildetwoneg{\mu_t}}^2 = \int_\Omega\,d\mu_t\,\left(\Delta \phi_{\dot{\mu}_t}\right)^2. \]

For $\psi \in \htildetwo{dx}$, we have that
\begin{align*}
C^{1/2}||\psi||_{\htildetwo{dx}}d_{\cdspacearg{\rho}}(\mu,\nu) &= C^{1/2}\sqrt{\int_\Omega\,dx\,(\Delta \psi)^2}\sqrt{\int_0^1\,dt\,\int_\Omega\,d\mu_t\,\left(\Delta \phi_{\dot{\mu}_t}\right)^2}\\
&\geq \sqrt{\int_0^1\,dt\,\int_\Omega\,d\mu_t\,(\Delta \psi)^2}\sqrt{\int_0^1\,dt\,\int_\Omega\,d\mu_t\,\left(\Delta \phi_{\dot{\mu}_t}\right)^2}\\
&\geq \int_0^1\,dt\,\int_\Omega\,d\mu_t\,\Delta\psi \Delta\phi_{\dot{\mu}_t}\\
&= \int_0^1\,dt\,\int_\Omega\,dx\,\psi \deltatwomu{\mu_t}\phi_{\dot{\mu}_t}\\
&= \int_0^1\,dt\,\int_\Omega\,dx\,\psi \dot{\mu}_t\\
&= \int_0^1\,dt\,\totdd{}{t}\int_\Omega\,dx\,\psi\mu_t\\
&= \langle \psi, (\nu-\mu)\rangle,
\end{align*}
where the first inequality made use of the assumption $\rho_t\leq C$ and the second follows from Cauchy-Schwarz. The duality pairing in the final line is between $\htildetwo{dx}$ and $\htildetwoneg{dx}.$

By re-arranging, we have \[C^{-1/2}\left\langle \frac{\psi}{||\psi||_{\htildetwo{dx}}},(\nu-\mu)\right\rangle\leq d_{\cdspacearg{\rho}}(\mu,\nu).\]

By taking a supremum over all $\psi \in \htildetwo{dx}$, we obtain \[C^{-1/2}||\nu-\mu||_{\htildetwoneg{dx}}\leq d_{\cdspacearg{\rho}}(\mu,\nu).\]

\end{proof}

In the classical $W_2$ theory, the result, \cite[Chapter~7]{MR3409718}, that the constant-speed geodesic $\mu_t$, between two measures with densities bounded by $C$, has density $\rho_t\leq C$ is proven using the geodesic convexity of the $||\cdot||_{L^q}^q$, but we do not develop notions of geodesic convexity in this paper.

\section{Conclusion}\label{sec:conclusion}

We have presented a new adaptation of the celebrated Benamou-Brenier expression of the $2$-Wasserstein metric, and shown that it provides an intuitive gradient flow formulation of a class of models arising in the field of econophysics.

Around 1998, it emerged that the $W_2$ metric allowed the heat equation $\partial_t \rho_t = \Delta \rho_t$ to be understood as equivalent to $\partial_t\rho_t = \grad{W_2}{\boltzmann{\rho_t}}$ where $\boltzmann{\rho}$ is the Boltzmann entropy. The novel result of \cref{thm:gradFlowEconophysics} is that the $\cdspace$ metric allows the continuum, diffusion approximations of a class of econophysics models to be viewed as gradient flows of the (scaled) Gini coefficient of economic inequality. 

For unbiased finite agent systems, the emergence of oligarchy is generally viewed as a consequence of martingale convergence theorems \cite{MR3475485} or simpler concentration phenomena \cite{borgers2023new}. Most previous results about the monotonicity of the Gini coefficient have been on a case-by-case basis except for the works \cite{MR4544046,MR4267576}. \Cref{prop:giniMonotoneEconophysics} gives a novel and simple proof of monotonicity that remains general enough to encompass the mean-field limits of the broad class of models satisfying \cref{def:wealthConserving,def:positivityPreserving,def:unbiased}. By doing so, the inexorable increase of inequality for these models, which can be considered a ``thermodynamic second law of econophysics,'' is established. 

Moreover, adhering to F. Otto's apothegm about delineating the kinetics of a system from its energetics allowed the intuitive and elegant gradient structure of this class of models to emerge. It is gratifying that once the kinetics of the collisional transaction were relegated to the metric tensor, the energetic functional that expressed each model as a gradient flow is a simple affine transformation of the classical Gini coefficient of economic inequality.

With \cref{thm:gradFlowEconophysics} it is now possible to see that these models may be viewed as maximally increasing inequality, as measured by the Gini coefficient, at each instant in time. In \cite{MR4961432}, casting the DLSS equation as a gradient flow in the diffusive transport metric yields a discretization that retains many essential structural properties of the continuum equation. This suggests that exploiting the gradient flow formulation in \cref{thm:gradFlowEconophysics} may likewise provide a fruitful basis for the numerical simulation of econophysics equations.

The work above is reminiscent of the search for noncanonical Poisson brackets for complex fluids \cite{MR1627532}. In which the Hamiltonian functionals could be more physically meaningful once the canonical brackets were jettisoned, in favor of those less standard. In this work, the Onsager operator and the free-energy functional generate the time evolution of the system whereas in Hamiltonian systems, it is the Poisson bracket and the Hamiltonian. The conserved quantities of \cref{lem:ConservedQuant} play a similar role to the Casimir invariants of Hamiltonian dynamics generated by a Poisson structure.

While the idea of econophysics motivated this new adaptation of the Wasserstein metric, the use of the theory is not restricted to econophysics alone. As part of presenting a more rounded study of these structures beyond econophysics, different transport inequalities were proven and the spaces $\htildetwo{\mu}$ and $\htildetwoneg{\mu}$ analyzed.

We end with a few questions appropriate for follow on study: Is there a signature of the fourth-order Onsager operator present in the discrete model, \cref{eq:microtransaction}? Is there a suitable notion of geodesic convexity in $\cdspace$ that would lead to convergence results? \cite{MR4987462} studies this question for $C_\rho$ in one spatial dimension. Does the $\cdspace$ gradient structure of the econophysics models lend itself to improved numerical methods? \cite{MR4961432} certainly suggests so. Does this metric structure have a static formulation, that is an interpretation through the original lens of optimal transport (push forwards, cost functions, etc.)? Martingale optimal transport, \cite{MR4003563}, nearly gives a static formulation for $C_\rho$. Can our construction be extended to discrete distributions over $\mathbb{N}$ by way of generalizing both the above work and the theory developed in \cite{MR3567821} and \cite{MR2824578}?

\begin{appendix}

\section{Proof of \texorpdfstring{\cref{prop:nosmoothfuncYSM}}{Proposition}}\label{apdx:noGoProof}

If there exists $F$ such that an operator $L = \frechet{F}{\rho}$ then we call $L$ a potential operator and $F$, its potential, is given by \cref{prop:frechetAnti}. This can be used to test the potentiality of operators in a constructive manner.

\begin{proposition}\label{prop:frechetAnti}
Let $F[\rho]$ be a functional with Fr\'echet derivative $\frechet{F}{\rho(x)}$ then \[F[\rho]=\int_0^1\,dr\,\int_\Omega \,dx\,\frechet{F[r\rho]}{ (r\rho) (x)}\rho(x) + C\] for some $C\in\mathbb{R}$.
\end{proposition}
This is a result that can be derived from the more general statement of the Fundamental Theorem of Calculus for the Gateaux derivative, but we prove it more simply here.

\begin{proof}
Fix $\rho$ in the domain of $F$ and define $f_\rho(r)=F[r\rho]$, then \begin{align*}
\totdd{f_\rho(r)}{r} &= \lim_{\epsilon \rightarrow 0}\frac{f_\rho(r+\epsilon)-f_\rho(r)}{\epsilon}\\
&= \lim_{\epsilon \rightarrow 0}\frac{F[r\rho+\epsilon \rho] - F[r\rho]}{\epsilon}\\
&= \int_\Omega\,dx\,\frechet{F[r\rho]}{(r\rho)(x)}\rho(x).
\end{align*}
Therefore \begin{align*}\int_0^1\,dr\,\totdd{f_\rho(r)}{r} &= f_\rho(1)-f_\rho(0) \\ &= F[\rho] -F[0] \\ &= \int_0^1\,dr\,\int_\Omega\,dx\,\frechet{F[r\rho]}{(r\rho)(x)}\rho(x). \end{align*}
\end{proof}

In this section, $\diff{w}{\rho}$ refers to the specific diffusion coefficient \begin{equation}\label{eq:ysmDiff}\frac{\gamma}{2}\int_{\mathbb{R}^+}\,dy\, (w\minop y)^2\rho_t(y)\end{equation} of the yard-sale model.

\begin{proposition}\label{prop:apdxNoGo}
Let $\Omega = \mathbb{R}^+$. There is no $C^1$ $F:\ptwoac \rightarrow \mathbb{R}$ such that the yard-sale model of \cref{eq:ysmPDE} is realized as the gradient system $\left(\ptwoac, \mathbb{K}_\text{Otto}, F\right)$.
\end{proposition}

\begin{proof}
Without loss of generality, set $\gamma = 1$. We seek to show that there does not exist a functional $F$ such that \[\wtwograd{F[\rho]}{w} = \twoddw{}\left(\diff{w}{\rho}\rho\right).\]

Our strategy is to assume the existence of $F$ so that we can attempt its construction by \cref{prop:frechetAnti} at which point the candidate functional will have its Fr\'echet derivative taken to see if it is the operator with which we began.

For simplicity, we incorporate the overall negative of the left hand side into the definition of $F$. Such a functional would satisfy \begin{equation}\label{eq:condOnF}\ddw{}\frechet{F[\rho]}{\rho} = \frac{1}{\rho}\left(\ddw{\left(\diff{w}{\rho}\rho\right)} + \alpha_1\right),\end{equation} where $\alpha_1 \in \mathbb{R}$.

Integrating \cref{eq:condOnF}, we have that \begin{equation}\label{eq:intCondOnF}\frechet{F[\rho]}{\rho} = \alpha_2 + \alpha_1 \int^w \,dy\, \frac{1}{\rho(y)} + \int^w\,dy\, \frac{1}{\rho}\ddy{\left(\diff{y}{\rho}\rho\right)},\end{equation} with $\alpha_2\in\mathbb{R}.$

Expanding the third term of \cref{eq:intCondOnF} and simplifying gives
\begin{align}
\frechet{F[\rho]}{\rho} &= \alpha_2 + \alpha_1 \int^w \,dy\, \frac{1}{\rho(y)} + \int^w\,dy\, \left[\diff{y}{\rho}\ddy{\log \rho} + \ddy{\diff{y}{\rho}}\right] \nonumber\\
&= \alpha_2 + \int^w \,dy\,\left[ \frac{\alpha_1 }{\rho(y)} +  \diff{y}{\rho}\ddy{\log \rho} \right]+ \diff{w}{\rho}.\label{eq:exprForFrechetF}
\end{align}

By Fr\'echet anti-differentiation of \cref{eq:exprForFrechetF} via \cref{prop:frechetAnti} we have a candidate functional $F$, 
\begin{align*}
F[\rho] &= \int_\Omega\,dw\,\rho(w) \left(\alpha_2 + \frac{ \diff{w}{\rho}}{2} \right) \\
&+\int_0^1\,dr\,\int_\Omega \,dw\,\left(\int^w \,dy\,\left[ \frac{\alpha_1 }{r\rho(y)} +  \diff{y}{r\rho}\ddy{\log r\rho} \right]\right)\rho(w)\\
&= \alpha_3 + \frac{1}{2}\int_\Omega\,dw\,\rho(w)\diff{w}{\rho}\\
&+ \alpha_1\int_0^1\,dr\,\frac{1}{r}\int_\Omega\,dw\,\rho(w)\int^w\,dy\,\frac{1}{\rho(y)}\\
&+ \frac{1}{2}\int_\Omega\,dw\,\rho(w)\int^w\,dy\, \diff{y}{\rho}\ddy{\log \rho}
\end{align*}
Now we must assume that $\alpha_1=0$ otherwise $F$ blows up. Thus what remains is \begin{equation}\label{eq:putativeFuncF}F[\rho] = \alpha_3 + \frac{1}{2}\int_\Omega\,dx\,\rho(w)\varfunc{w}{\rho}{D} + \frac{1}{2}\int_\Omega\,dw\,\rho(w)\int^w\,dy\, \diff{y}{\rho}\ddy{\log \rho}\end{equation} as a candidate functional for the operator \begin{equation}\label{eq:nonPotOpL}\alpha_2 + \diff{w}{\rho} + \int^w \,dy\, \diff{y}{\rho}\ddy{\log \rho} .\end{equation}

We now compute the Fr\'echet derivative of the candidate functional \cref{eq:putativeFuncF} to see if the operator from which we began, \cref{eq:nonPotOpL}, is in fact potential.

Taking the Fr\'echet derivative gives \[\frechet{F}{\rho} = \diff{w}{\rho} + \frechet{}{\rho} \underbrace{\frac{1}{2}\int_\Omega\,dw\,\rho(w)\int^w\,dy\, \diff{y}{\rho}\ddy{\log \rho}}_{=:F_L}.\]

The first term of $\frechet{F}{\rho}$ is the second term of \cref{eq:nonPotOpL},  so we work on the second term to see if $\frechet{F_L}{\rho}$ is \[L:=\int^w \,dy\,\diff{y}{\rho}\ddy{\log \rho},\] the third term of \cref{eq:nonPotOpL}.

By choosing a lower integration bound of $0$ in $F_L$ and then changing the order of integration and integrating by parts \[F_L = \frac{1}{2} \int_\Omega\,dy\,\log{\rho(y)}\left[\diff{y}{\rho}\rho(y) - y \left(\int_y^\infty\,dw\,\rho(w)\right)^2\right].\]

The Fr\'echet derivative of $F_L$ is \[\frac{1}{2} \left[\diff{w}{\rho\log\rho} + \left(1+\log\rho(w)\right)\diff{w}{\rho}-\frac{w}{\rho(w)}\left(\int_w^\infty\,dx\,\rho(x)\right)^2\right] - \int_0^w\,dx\,x\log\rho(x)\int_x^\infty\,dy\,\rho(y)\] and note that the last term is precisely $L$ by integrating by parts after noting that \[x\int_x^\infty\,dy\,\rho(y)=\ddx{}\diff{x}{\rho} .\]

In which case we have that \[0\neq\frechet{F_L}{\rho}-L = \frac{1}{2} \left[\diff{w}{\rho\log\rho} + \left(1+\log\rho(w)\right)\diff{w}{\rho}-\frac{w}{\rho(w)}\left(\int_w^\infty\,dx\,\rho(x)\right)^2\right].\]

And therefore by contrapositive we have that $L$ and \cref{eq:nonPotOpL} are not potential operators. This makes the entire construction fail since the candidate functional $F$ does not have \cref{eq:nonPotOpL} as its Fr\'echet derivative, thus proving the claim.
\end{proof}

\begin{question*}
Let $\diff{w}{\rho}$ be given by \cref{eq:ysmDiff}, the specific diffusion operator for the yard-sale model. Can $m: \mathbb{R}^+ \rightarrow \mathbb{R}^+$ be found so that \begin{equation}\alpha_2+\int^w\,dy\,\left\{\frac{1}{m(\rho(y))}\left[\alpha_1+\ddy{\left(\rho\diff{y}{\rho}\right)}\right]\right\}\label{eq:possiblePotOpYSM}\end{equation} is a potential operator (i.e., the Fr\'echet derivative of a functional)?
\end{question*}

We have seen by \cref{prop:apdxNoGo} that for $m=Id$ in \cref{eq:possiblePotOpYSM}, we can answer this question in the negative.

A systematic method of answering this question is not clear though routes may be suggested in \cite{MR0458271, MR0544433, MR0666032, MR0769972, MR0088708}. Finding an appropriate mobility function in order to ``reverse engineer'' a potential operator feels very much like a functional version of the method of integrating factors in differential equations.

Two natural guesses are $m(\rho) = \left(\rho \diff{w}{\rho}\right)$ and $m(\rho) = \left(\rho \diff{w}{\rho}\right)^{-1}$ so that some simplification is afforded in \cref{eq:possiblePotOpYSM}. These two do not work and we omit the proofs as they are in much the same spirit as the proof of \cref{prop:nosmoothfuncYSM}. A constructive method for $m$ that guarantees the potentiality of the operator \cref{eq:possiblePotOpYSM} would be quite valuable.

We have no proof that there is no nonnegative $m$ that makes \cref{eq:possiblePotOpYSM} potential, but there is worth in considering heuristically \textit{why} the Onsager operator $\kotto$ for $W_2$ is a second-order differential operator. We do so in \cref{apdx:YSMW2mDiscussion} and suggest that a fourth-order Onsager operator is suitable for the problems that we wish to treat.

\section{Modeling heuristics in favor of fourth-order Onsager operator}\label{apdx:YSMW2mDiscussion}

The discussion is given in one dimension as it serves to motivate \cref{ssec:gradFlowKaem}.

We can gain insight by first considering the time-independent equilibrium and then retroactively coming up with a time-dependent explanation of how the system arrived at equilibrium.

\subsubsection{Optimization with conservation of probability (mass)}
Let $I\subseteq\mathbb{R}$ be an interval with endpoints $a<b$. The evolution equations that classical and nonlinear mobility $2$-Wasserstein gradient flow theory treat are those with, in general, one conserved quantity: total probability. That is, let $F:\ptwoacarg{I} \rightarrow \mathbb{R}$ be a functional then recalling nonlinear mobility $2$-Wasserstein Otto calculus, \[\ddt{\rho} = \grad{W_{2,m}}{F}=\wtwoMobgrad{F}{m(\rho)}{x}\] is such that (1) $N[\rho]:=\int_I\,dx\,\rho_t(x) = 1$ for all $t$ and (2) $t\mapsto F[\rho_t]$ is an increasing function.

In which case the asymptotic behavior both extremizes $F$ and maintains the conservation of total probability. Instead of starting with the time-dependent evolution equation, we could first consider the constrained optimization problem
\begin{center}
    extremize $F[\rho]$ subject to $N[\rho] = 1$.
\end{center}
The method of Lagrange multipliers would have us solve for a solution $\rho^*$ such that \begin{align}
\frechet{F[\rho^*]}{\rho} &= \lambda \frechet{N[\rho^*]}{\rho} \nonumber \\ &= \lambda \label{eq:lagrangeMult1}, \end{align} where $\lambda \in \mathbb{R}$ is the Lagrange multiplier associated with the constraint $N[\rho]=1$. 

To construct an evolution equation such that a subset of its stationary solutions coincide with the solutions $\rho^*$ of the constrained optimization problem, consider \begin{equation*} \ddt{\rho} = L\left[\ddx{}\frechet{F}{\rho}\right],\end{equation*} where the operator $L$ has the trivial function in its kernel.

In this case, \begin{equation*} \ddt{\rho^*} = L\left[\ddx{}\frechet{F[\rho^*]}{\rho}\right]=L[0]=0\end{equation*} for a solution $\rho^*$ of \cref{eq:lagrangeMult1}, the constrained optimization problem.

Now let \begin{equation*} L[\nu] = -\ddx{}\left[\varfunc{x}{\rho}{H}\nu\right], \end{equation*} where $\varfunc{x}{\rho}{H}$ is (without loss of generality) non-negative. This additional structure makes $F$ monotone under the dynamics since \begin{equation*}\ddt{\rho} = -\ddx{}\left[\varfunc{x}{\rho}{H}\ddx{}\frechet{F}{\rho}\right]\end{equation*} so that \begin{align*} \totdd{F}{t} &= \int_I\,dx\,\frechet{F}{\rho}\ddt{\rho} \\
&= + \int_I\,dx\,\varfunc{x}{\rho}{H}\left(\ddx{}\frechet{F}{\rho}\right)^2\\
&\geq 0.
\end{align*} Moreover, this choice of $L$ means that the dynamics stop if \cref{eq:lagrangeMult1} is satisfied since $-\ddx{}\left[\varfunc{x}{\rho}{H}\ddx{\lambda}\right] = 0$ for $\lambda \in \mathbb{R}$. Note the assumed sufficient condition for the monotonicity of $F$ is the vanishing of the boundary terms\footnote{For a more involved discussion on these surface integrals in higher dimensions, see \cite[Section~4.3]{MR2118870}.} \[\left(\frechet{F}{\rho}\varfunc{x}{\rho}{H}\ddx{}\frechet{F}{\rho}\right)\bigg|_{x=a}^{x=b} = 0.\]

The way in which we chose $L$ and $\varfunc{x}{\rho}{H}$ was to make the overall operator acting on $\frechet{F}{\rho}$ be both symmetric and positive. This procedure is elegantly presented in \cite{mielke2023introduction}.

Requiring the conservation of $N[\rho]$ under the dynamics demands that  \[\left(\varfunc{x}{\rho}{H}\ddx{}\frechet{F}{\rho}\right)\bigg|_{x=a}^{x=b} = 0.\]

Evolution equations constructed in this way are not necessarily the only evolution equation such that (1) a subset of their stationary solutions coincide with the solutions of the constrained optimization problem and (2) $F$ is monotone under the evolutionary dynamics. Rather these are simple choices that encompass several important examples. 

The choice of $\varfunc{x}{\rho}{H} = \rho$ gives the linear mobility free-energy Fokker-Planck equation (in the language of \cite{MR2118870}). This is the case of the classical 2-Wasserstein gradient operator, $\kotto$.  If $\varfunc{x}{\rho}{\mathcal{M}}$ is non-negative then $\varfunc{x}{\rho}{H} =\varfunc{x}{\rho}{\mathcal{M}}\rho$ corresponds to the case of non-trivial, time-homogeneous mobility coefficients in \cite{MR2118870}. If $\varfunc{x}{\rho}{H}$ is set to a positive constant, then we have the Cahn-Hilliard model.

\subsubsection{Optimization with conservation of probability (mass) and first moment}
If it is known \textit{a priori} that the first moment is also a conserved quantity, we should instead consider the constrained optimization problem
\begin{equation*}
\begin{cases}
\text{extremize } &F[\rho]\\
\text{subject to } &1=N[\rho]:=\int_I\,dx\,\rho(x)\\
\text{ \quad \quad and } &\mu = W[\rho]:=\int_I\,dx\,x\rho(x).
\end{cases}
\end{equation*}

The method of Lagrange multipliers would have us solve for $\rho^*$ such that \begin{align}
0& = \frechet{F[\rho*]}{\rho} - \lambda \frechet{N[\rho^*]}{\rho} - \psi \frechet{W[\rho*]}{\rho} \nonumber\\
&= \frechet{F[\rho^*]}{\rho} -\lambda - \psi x, \label{eq:lagrangeMult2}
\end{align} where $\lambda, \psi \in \mathbb{R}$ are the Lagrange multipliers for the constraints $N[\rho]$ and $W[\rho]$, respectively. That is, the Fr\'echet derivative at the solution is an affine function.

Therefore for some operator $J$ such that $0 \in \ker(J)$, \[\ddt{\rho} = J\left[\twoddx{}\left(\frechet{F}{\rho}\right)\right]\] will cease evolving when $\frechet{F}{\rho}$ is affine. To symmetrize and make positive definite the operator acting on $\frechet{F}{\rho}$, let \[J[\nu] = \twoddx{}\left[\varfunc{x}{\rho}{H}\nu\right],\] for $\varfunc{x}{\rho}{H}\geq 0$. These choices imply $J[0] = 0$ and $\totdd{F}{t} \geq 0.$

Therefore \begin{equation}\ddt{\rho} = \twoddx{}\left[\varfunc{x}{\rho}{H}\twoddx{}\frechet{F}{\rho}\right]\label{eq:dcfeevoeq}\end{equation} describes an evolutionary system for which $F$ increases monotonically, and if both \[\left(\ddx{}\left[\varfunc{x}{\rho}{H}\twoddx{}\frechet{F}{\rho}\right]\right)\bigg|_{x=a}^{x=b}\] and
\[\left(x\ddx{}\left[\varfunc{x}{\rho}{H}\twoddx{}\frechet{F}{\rho}\right] - \varfunc{x}{\rho}{H}\twoddx{}\frechet{F}{\rho}\right)\bigg|_{x=a}^{x=b}\] vanish then the zeroth and first moments are conserved.

Whereas the case of conservation of only probability mass gave rise to the second-order operator $\kotto$, the additional conserved quantity of first moment suggests a fourth-order operator is more appropriate to consider.

If in \cref{eq:dcfeevoeq}, $\varfunc{x}{\rho}{H}$ is an integral diffusion coefficient from a kinetic asset exchange model and the functional $F$ is the scaled Gini coefficient $G$ then we have exactly the structure seen in \cref{eq:kaemDiffGini}.

\end{appendix}

\bibliographystyle{amsplain}
\bibliography{Bilbiography}

\providecommand{\bysame}{\leavevmode\hbox to3em{\hrulefill}\thinspace}
\providecommand{\MR}{\relax\ifhmode\unskip\space\fi MR }
\providecommand{\MRhref}[2]{%
  \href{http://www.ams.org/mathscinet-getitem?mr=#1}{#2}
}
\providecommand{\href}[2]{#2}
\begin{thebibliography}{10}

\bibitem{MR4294651}
Luigi Ambrosio, Elia Bru\'{e}, and Daniele Semola, \emph{Lectures on optimal transport}, Unitext, vol. 130, Springer, Cham, 2021, La Matematica per il 3+2. \MR{4294651}

\bibitem{MR0458271}
R.~W. Atherton and G.~M. Homsy, \emph{On the existence and formulation of variational principles for nonlinear differential equations}, Studies in Appl. Math. \textbf{54} (1975), no.~1, 31--60. \MR{458271}

\bibitem{MR2604625}
Marzia Bisi, Giampiero Spiga, and Giuseppe Toscani, \emph{Kinetic models of conservative economies with wealth redistribution}, Commun. Math. Sci. \textbf{7} (2009), no.~4, 901--916. \MR{2604625}

\bibitem{MR3443169}
Vladimir~I. Bogachev, Nicolai~V. Krylov, Michael R\"ockner, and Stanislav~V. Shaposhnikov, \emph{Fokker-{P}lanck-{K}olmogorov equations}, Mathematical Surveys and Monographs, vol. 207, American Mathematical Society, Providence, RI, 2015. \MR{3443169}

\bibitem{MR4686618}
Bruce Boghosian and Christoph B\"{o}rgers, \emph{The mathematics of poverty, inequality, and oligarchy}, SIAM News \textbf{56} (2023), no.~8, 1--2. \MR{4686618}

\bibitem{MR3623598}
Bruce~M. Boghosian, Adrian Devitt-Lee, Merek Johnson, Jie Li, Jeremy~A. Marcq, and Hongyan Wang, \emph{Oligarchy as a phase transition: the effect of wealth-attained advantage in a {F}okker-{P}lanck description of asset exchange}, Phys. A \textbf{476} (2017), 15--37. \MR{3623598}

\bibitem{MR3428664}
Bruce~M. Boghosian, Merek Johnson, and Jeremy~A. Marcq, \emph{An {$H$} theorem for {B}oltzmann's equation for the yard-sale model of asset exchange. {T}he {G}ini coefficient as an {$H$} functional}, J. Stat. Phys. \textbf{161} (2015), no.~6, 1339--1350. \MR{3428664}

\bibitem{borgers2023new}
Christoph B\"orgers and Claude Greengard, \emph{A new probabilistic analysis of the yard-sale model}, 2023.

\bibitem{MR4544046}
Ben-Hur~Francisco Cardoso, Sebasti\'{a}n Gon\c{c}alves, and Jos\'{e}~Roberto Iglesias, \emph{Why equal opportunities lead to maximum inequality? {T}he wealth condensation paradox generally solved}, Chaos, Solitons, \& Fractals \textbf{168} (2023), Paper No. 113181, 5. \MR{4544046}

\bibitem{MR4267576}
Ben-Hur~Francisco Cardoso, Jos\'{e}~Roberto Iglesias, and Sebasti\'{a}n Gon\c{c}alves, \emph{Wealth concentration in systems with unbiased binary exchanges}, Phys. A \textbf{579} (2021), Paper No. 126123, 8. \MR{4267576}

\bibitem{MR2565840}
J.~A. Carrillo, S.~Lisini, G.~Savar\'{e}, and D.~Slep\v{c}ev, \emph{Nonlinear mobility continuity equations and generalized displacement convexity}, J. Funct. Anal. \textbf{258} (2010), no.~4, 1273--1309. \MR{2565840}

\bibitem{MR4746872}
Jos\'e{}~A. Carrillo, Matias~G. Delgadino, Laurent Desvillettes, and Jeremy~S.H. Wu, \emph{The {L}andau equation as a gradient flow}, Anal. PDE \textbf{17} (2024), no.~4, 1331--1375. \MR{4746872}

\bibitem{AC2002}
Anirban Chakroborti, \emph{{Distributions of Money in Model Markets of Economy}}, International Journal of Modern Physics C \textbf{13} (2002), no.~10, 1315--1321.

\bibitem{MR3475485}
Christophe Chorro, \emph{A simple probabilistic approach of the yard-sale model}, Statist. Probab. Lett. \textbf{112} (2016), 35--40. \MR{3475485}

\bibitem{MR4804189}
David~W. Cohen and Bruce~M. Boghosian, \emph{Bounding the {A}pproach to {O}ligarchy in a {V}ariant of the {Y}ard-{S}ale {M}odel}, SIAM J. Appl. Math. \textbf{84} (2024), no.~5, 2051--2066. \MR{4804189}

\bibitem{MR2448650}
Jean Dolbeault, Bruno Nazaret, and Giuseppe Savar\'{e}, \emph{A new class of transport distances between measures}, Calc. Var. Partial Differential Equations \textbf{34} (2009), no.~2, 193--231. \MR{2448650}

\bibitem{MR2551376}
Bertram D\"{u}ring, Daniel Matthes, and Giuseppe Toscani, \emph{Kinetic equations modelling wealth redistribution: a comparison of approaches}, Phys. Rev. E (3) \textbf{78} (2008), no.~5, 056103, 12. \MR{2551376}

\bibitem{MR3567821}
Matthias Erbar, Max Fathi, Vaios Laschos, and Andr\'e Schlichting, \emph{Gradient flow structure for {M}c{K}ean-{V}lasov equations on discrete spaces}, Discrete Contin. Dyn. Syst. \textbf{36} (2016), no.~12, 6799--6833. \MR{3567821}

\bibitem{MR2597943}
Lawrence~C. Evans, \emph{Partial differential equations}, second ed., Graduate Studies in Mathematics, vol.~19, American Mathematical Society, Providence, RI, 2010. \MR{2597943}

\bibitem{MR2118870}
Till~Daniel Frank, \emph{Nonlinear {F}okker-{P}lanck equations}, Springer Series in Synergetics, Springer-Verlag, Berlin, 2005, Fundamentals and applications. \MR{2118870}

\bibitem{MR4331350}
Marek Galewski, \emph{Basic monotonicity methods with some applications}, Compact Textbooks in Mathematics, Birkh\"{a}user/Springer, Cham, 2021. \MR{4331350}

\bibitem{MR2667016}
Filippo Gazzola, Hans-Christoph Grunau, and Guido Sweers, \emph{Polyharmonic boundary value problems: Positivity preserving and nonlinear higher order elliptic equations in bounded domains}, Lecture Notes in Mathematics, vol. 1991, Springer-Verlag, Berlin, 2010. \MR{2667016}

\bibitem{MR2895086}
N.~Gozlan and C.~L\'eonard, \emph{Transport inequalities. {A} survey}, Markov Process. Related Fields \textbf{16} (2010), no.~4, 635--736. \MR{2895086}

\bibitem{BH2002}
Brian Hayes, \emph{{Follow the Money}}, American Scientist \textbf{90} (2002), no.~5, 400--405.

\bibitem{MR4003563}
Martin Huesmann and Dario Trevisan, \emph{A {B}enamou-{B}renier formulation of martingale optimal transport}, Bernoulli \textbf{25} (2019), no.~4A, 2729--2757. \MR{4003563}

\bibitem{MR4185105}
Jos\'{e}~Roberto Iglesias, Ben-Hur~Francisco Cardoso, and Sebasti\'{a}n Gon\c{c}alves, \emph{Inequality, a scourge of the {XXI} century}, Commun. Nonlinear Sci. Numer. Simul. \textbf{95} (2021), Paper No. 105646, 8. \MR{4185105}

\bibitem{MR1617171}
Richard Jordan, David Kinderlehrer, and Felix Otto, \emph{The variational formulation of the {F}okker-{P}lanck equation}, SIAM J. Math. Anal. \textbf{29} (1998), no.~1, 1--17. \MR{1617171}

\bibitem{MR735549}
Allan~N. Kaufman, \emph{Dissipative {H}amiltonian systems: a unifying principle}, Phys. Lett. A \textbf{100} (1984), no.~8, 419--422. \MR{735549}

\bibitem{MR3872473}
Jie Li, Bruce~M. Boghosian, and Chengli Li, \emph{The affine wealth model: an agent-based model of asset exchange that allows for negative-wealth agents and its empirical validation}, Phys. A \textbf{516} (2019), 423--442. \MR{3872473}

\bibitem{MR4222140}
Wuchen Li, \emph{Hessian metric via transport information geometry}, J. Math. Phys. \textbf{62} (2021), no.~3, Paper No. 033301, 23. \MR{4222140}

\bibitem{MR3150642}
Matthias Liero and Alexander Mielke, \emph{Gradient structures and geodesic convexity for reaction-diffusion systems}, Philos. Trans. R. Soc. Lond. Ser. A Math. Phys. Eng. Sci. \textbf{371} (2013), no.~2005, 20120346, 28. \MR{3150642}

\bibitem{MR2672546}
Stefano Lisini and Antonio Marigonda, \emph{On a class of modified {W}asserstein distances induced by concave mobility functions defined on bounded intervals}, Manuscripta Math. \textbf{133} (2010), no.~1-2, 197--224. \MR{2672546}

\bibitem{MR2921215}
Stefano Lisini, Daniel Matthes, and Giuseppe Savar\'{e}, \emph{Cahn-{H}illiard and thin film equations with nonlinear mobility as gradient flows in weighted-{W}asserstein metrics}, J. Differential Equations \textbf{253} (2012), no.~2, 814--850. \MR{2921215}

\bibitem{MR3558359}
Daniel Loibl, Daniel Matthes, and Jonathan Zinsl, \emph{Existence of weak solutions to a class of fourth order partial differential equations with {W}asserstein gradient structure}, Potential Anal. \textbf{45} (2016), no.~4, 755--776. \MR{3558359}

\bibitem{MR2824578}
Jan Maas, \emph{Gradient flows of the entropy for finite {M}arkov chains}, J. Funct. Anal. \textbf{261} (2011), no.~8, 2250--2292. \MR{2824578}

\bibitem{MR4179806}
Jan Maas and Alexander Mielke, \emph{Modeling of chemical reaction systems with detailed balance using gradient structures}, J. Stat. Phys. \textbf{181} (2020), no.~6, 2257--2303. \MR{4179806}

\bibitem{MR2581977}
Daniel Matthes, Robert~J. McCann, and Giuseppe Savar\'{e}, \emph{A family of nonlinear fourth order equations of gradient flow type}, Comm. Partial Differential Equations \textbf{34} (2009), no.~10-12, 1352--1397. \MR{2581977}

\bibitem{MR4961432}
Daniel Matthes, Eva-Maria Rott, Giuseppe Savar\'e, and Andr\'e Schlichting, \emph{A structure preserving discretization for the {D}errida-{L}ebowitz-{S}peer-{S}pohn equation based on diffusive transport}, Numer. Math. \textbf{157} (2025), no.~4, 1347--1395. \MR{4961432}

\bibitem{MR4987462}
Daniel Matthes, Eva-Maria Rott, and Andr\'e Schlichting, \emph{Diffusive transport on the real line: semi-contractive gradient flows and their discretization}, Nonlinearity \textbf{38} (2025), no.~11, Paper No. 115005. \MR{4987462}

\bibitem{MR2777530}
Vladimir Maz'ya, \emph{Sobolev spaces with applications to elliptic partial differential equations}, augmented ed., Fundamental Principles of Mathematical Sciences, vol. 342, Springer, Heidelberg, 2011. \MR{2777530}

\bibitem{MR2776123}
Alexander Mielke, \emph{A gradient structure for reaction-diffusion systems and for energy-drift-diffusion systems}, Nonlinearity \textbf{24} (2011), no.~4, 1329--1346. \MR{2776123}

\bibitem{mielke2023introduction}
Alexander Mielke, \emph{An introduction to the analysis of gradients systems}, 2023.

\bibitem{MR1627532}
P.~J. Morrison, \emph{Hamiltonian description of the ideal fluid}, Rev. Modern Phys. \textbf{70} (1998), no.~2, 467--521. \MR{1627532}

\bibitem{PhysRevA.33.4205}
P.~J. Morrison and S.~Eliezer, \emph{Spontaneous symmetry breaking and neutral stability in the noncanonical hamiltonian formalism}, Phys. Rev. A \textbf{33} (1986), 4205--4214.

\bibitem{MR57765}
L.~Onsager and S.~Machlup, \emph{Fluctuations and irreversible processes}, Phys. Rev. (2) \textbf{91} (1953), 1505--1512. \MR{57765}

\bibitem{HCOthermodynamics}
Hans~Christian \"Ottinger, \emph{Beyond equilibrium thermodynamics}, John Wiley and Sons, Ltd, 2005.

\bibitem{MR1842429}
Felix Otto, \emph{The geometry of dissipative evolution equations: the porous medium equation}, Comm. Partial Differential Equations \textbf{26} (2001), no.~1-2, 101--174. \MR{1842429}

\bibitem{MR3922440}
R\'emi Peyre, \emph{Comparison between {$\rm W_2$} distance and {$\dot {\rm H}^{-1}$} norm, and localization of {W}asserstein distance}, ESAIM Control Optim. Calc. Var. \textbf{24} (2018), no.~4, 1489--1501. \MR{3922440}

\bibitem{MR2155255}
Przemys{\l}aw Repetowicz, Stefan Hutzler, and Peter Richmond, \emph{Dynamics of money and income distributions}, Phys. A \textbf{356} (2005), no.~2-4, 641--654. \MR{2155255}

\bibitem{MR3014456}
Tom\'{a}\v{s} Roub\'{\i}\v{c}ek, \emph{Nonlinear partial differential equations with applications}, second ed., International Series of Numerical Mathematics, vol. 153, Birkh\"{a}user/Springer Basel AG, Basel, 2013. \MR{3014456}

\bibitem{MR3409718}
Filippo Santambrogio, \emph{Optimal transport for applied mathematicians: Calculus of variations, pdes, and modeling}, Progress in Nonlinear Differential Equations and their Applications, vol.~87, Birkh\"auser/Springer, Cham, 2015. \MR{3409718}

\bibitem{MR0544433}
J.~J. Telega, \emph{On variational formulations for nonlinear, nonpotential operators}, J. Inst. Math. Appl. \textbf{24} (1979), no.~2, 175--195. \MR{544433}

\bibitem{MR0666032}
J\'{o}zef~Joachim Telega, \emph{Determination of the potential form of operators}, Mat. Stos. (3) \textbf{18} (1982), 107--122. \MR{666032}

\bibitem{MR0769972}
Enzo Tonti, \emph{Variational formulation for every nonlinear problem}, Internat. J. Engrg. Sci. \textbf{22} (1984), no.~11-12, 1343--1371. \MR{769972}

\bibitem{MR0088708}
M.~M. Va{\u{\i}}nberg, \emph{Variational methods for investigation of non-linear operators}, Gosudarstv. Izdat. Tehn.-Teor. Lit., Moscow, 1956. \MR{88708}

\bibitem{Vainberg1964}
Mordukhai~Moiseevich Vainberg, \emph{Variational methods for the study of nonlinear operators}, Holden-Day Series in Mathematical Physics, Holden-Day, 1964.

\bibitem{valli2023wellposed}
Alberto Valli, \emph{A well-posed variational formulation of the {N}eumann boundary value problem for the biharmonic operator}, 2023.

\bibitem{MR1964483}
C\'{e}dric Villani, \emph{Topics in optimal transportation}, Graduate Studies in Mathematics, vol.~58, American Mathematical Society, Providence, RI, 2003. \MR{1964483}

\bibitem{MR2459454}
\bysame, \emph{Optimal transport}, Fundamental Principles of Mathematical Sciences, vol. 338, Springer-Verlag, Berlin, 2009, Old and new. \MR{2459454}

\bibitem{MR3012052}
Shlomo Yitzhaki and Edna Schechtman, \emph{The {G}ini methodology: A primer on a statistical methodology}, Springer Series in Statistics, Springer, New York, 2013. \MR{3012052}

\end{thebibliography}
\end{document}